\documentclass[10pt,twoside,english,reqno,a4paper]{amsproc}

\usepackage{geometry}
\newlength{\defbaselineskip}
\usepackage[bottom]{footmisc}
\makeatletter
\def\blfootnote{\gdef\@thefnmark{}\@footnotetext}
\makeatother

\usepackage{listings,graphicx,amsmath,varioref,amscd,amssymb,color,bm,stmaryrd}
\usepackage{mathpazo} 
\usepackage[scaled=.95]{helvet} 
\usepackage{courier}
\usepackage{amssymb}
\usepackage{amsmath}
\usepackage{esint}
\usepackage{amsthm}
\usepackage{amscd}
\usepackage{amssymb}
\usepackage{amsfonts}
\usepackage[english]{babel}
\usepackage[latin1]{inputenc}
\usepackage{booktabs}
\usepackage{datetime}

\definecolor{caribbeangreen}{rgb}{0.0, 0.8, 0.6}
\definecolor{darkpastelgreen}{rgb}{0.01, 0.75, 0.24}
\definecolor{green(pigment)}{rgb}{0.0, 0.65, 0.31}
\usepackage[dvipsnames]{xcolor}
\usepackage{hyperref} 
\hypersetup{
colorlinks=true,
linkcolor= red!70!black,
citecolor= orange,
urlcolor=orange, 
}

\addto\extrasenglish{%
}

\addto\extrasenglish{%
}

\addto\extrasenglish{%
}

\usepackage[capitalize,nameinlink,noabbrev,nosort]{cleveref} 

\crefname{section}{\sc Section}{\sc Sections}
\crefname{subsection}{\sc Subsection}{\sc Subsections}
\crefname{appendix}{\sc Appendix}{\sc Appendices}

\crefname{figure}{\sc Figure}{\sc Figures}

\crefname{definition}{\sc Definition}{\sc Definitions}
\crefname{theorem}{\sc Theorem}{\sc Theorems}
\crefname{proposition}{\sc Proposition}{\sc Propositions}
\crefname{corollary}{\sc Corollary}{\sc Corollaries}
\crefname{remark}{\sc Remark}{\sc Remarks}
\crefname{lemma}{\sc Lemma}{\sc Lemmata}
\crefname{theoa}{\sc Theorem}{\sc Theorems}
\crefname{lemb}{\sc Lemma}{\sc Lemmata}

\usepackage[titletoc,toc,title]{appendix}

\usepackage{graphics}
\usepackage{pgfplots}
\pgfplotsset{width=7cm,compat=newest} 
\usepackage{graphics}
\usepackage{tikz}
\usetikzlibrary{shapes.geometric}
\usepackage{tikz-cd}
\usetikzlibrary{patterns}
\usepackage{caption}
\usepackage{subcaption}

\usepackage{float}

\usepackage{pgfplotstable}
\usetikzlibrary{intersections}

\newtheoremstyle{mytheoremstyle} 
{.3em}                    
{0cm}                    
{\slshape}                   
{}                           
{\color{
red!70!black}\large\scshape}             
{.}                          
{.5em}                       
{}  

\theoremstyle{mytheoremstyle}

\newtheorem{theorem}{Theorem}[section]
\newtheorem{proposition}[theorem]{Proposition}

\newtheorem{lemma}[theorem]{Lemma}

\newtheorem{corollary}[theorem]{Corollary}

\newtheorem{definition}[theorem]{Definition}

\newtheorem{remark}[theorem]{Remark}

\renewcommand{\theequation}{\thesection.\arabic{equation}}
\numberwithin{equation}{section}
\long\def\salta#1{\relax}

\makeatletter

\def\b{\beta}
\def\N{\nabla}
\def\ga{\gamma}
\def\vp{\varphi}
\def\al{\alpha}
\def\om{\omega}
\def\vare{\varepsilon}
\def\eps{\varepsilon}

\def\ro{\rho}
\def\te{\theta}

\newcommand{\de}{\delta}
\newcommand{\lm}{\lambda}
\newcommand{\si}{\sigma}
\newcommand{\D}{\Delta}
\newcommand{\car}[1]{\raise1pt\hbox{$\chi$}_{#1}}
\def\q{\quad}
\def\qq{\qquad}
\def\t{\text}

\newcommand{\integrale}{\int_\Omega}

\newcommand{\ds}{\displaystyle}

\def\dive{\text{div }}
\DeclareMathOperator{\R}{\mathbb{R}}

\def\og{\leavevmode\raise.3ex\hbox{$\scriptscriptstyle\langle\!\langle$~}}
\def\fg{\leavevmode\raise.3ex\hbox{~$\!\scriptscriptstyle\,\rangle\!\rangle$}}

\definecolor{bor}{cmyk}{0.21,0.93,0.86,0.12}
\definecolor{air}{rgb}{0.178, 0.51, 0.51}
\definecolor{air2}{cmyk}{.82, 0., .67, .11}
\definecolor{range}{cmyk}{0,0.599,1,0.188}
\definecolor{range2}{rgb}{0.599,1,0.188}
\definecolor{vio}{rgb}{.4,0,.4}
\definecolor{pan}{rgb}{.17,.87,.64}
\definecolor{seagreen}{rgb}{0.18, 0.55, 0.34}

\def\eqref#1{(\ref{#1})}

\def\theequation{\arabic{section}.\arabic{equation}}

\def\be{\begin{equation}}
\def\ee{\end{equation}}

\def\beac{\be\begin{array}{c}}
\def\eeac{\end{array}\ee}

\def\theequation{\arabic{section}.\arabic{equation}}

\def\og{\leavevmode\raise.3ex\hbox{$\scriptscriptstyle\langle\!\langle$~}}
\def\fg{\leavevmode\raise.3ex\hbox{~$\!\scriptscriptstyle\,\rangle\!\rangle$}}

\author[M. Magliocca]{Martina Magliocca}

\address[M. Magliocca]{Dipartimento di Matematica, Universit\`a degli Studi Tor Vergata, Via della Ricerca Scientifica 1, 00133 Rome, Italy. 
\\ \url{magliocc@mat.uniroma2.it}}


\begin{document}


\title[Regularity and decay results of a parabolic problem]{Regularizing effect and decay results for a parabolic problem with repulsive superlinear first order terms}

\begin{abstract}
We want to analyse both regularizing effect and long, short time decay concerning parabolic Cauchy-Dirichlet problems of the type
\begin{equation*}
\begin{cases}
\begin{array}{ll}
u_t-\dive (A(t,x)|\N u|^{p-2}\N u)=\ga |\N u|^q & \text{in}\,\,Q_T,\\ 
u=0  &\text{on}\,\,(0,T)\times\partial\Omega,\\
u(0,x)=u_0(x) &\text{in}\,\, \Omega.
\end{array}
\end{cases}
\end{equation*}
We assume that $A(t,x)$ is a coercive, bounded and measurable matrix, the growth rate $q$ of the gradient term is superlinear but still subnatural, $\ga>0$, the initial datum $u_0$ is an unbounded function belonging to a well precise Lebesgue space $L^\si(\Omega)$ for $\si=\si(q,p,N)$.
\end{abstract}

\maketitle
\tableofcontents

\section{Introduction} 

\setcounter{equation}{0}
\renewcommand{\theequation}{\thesection.\arabic{equation}}

\numberwithin{equation}{section}

The main goal of this work is proving regularity and decay results regarding {solutions} of a class of parabolic equations with superlinear (and subquadratic) growth. \\ The model we consider is the following:
\begin{equation}\label{pb}
\begin{cases}
\begin{array}{ll}
u_t-\dive (A(t,x)|\N u|^{p-2}\N u)=\ga |\N u|^q & \text{in}\,\,Q_T,\\ 
u=0  &\text{on}\,\,(0,T)\times\partial\Omega,\\
u(0,x)=u_0(x) &\text{in}\,\, \Omega,
\end{array}
\end{cases}
\end{equation}
where $\Omega$ is a bounded subset of $\Omega\subset\R^N$, $N\ge 2$, $Q_T=(0,T)\times \Omega$ is the parabolic cylinder, $1<p<N$ and $q<p$.\\
The problem in \eqref{pb} collects all the basic features which motivate our incoming study. Let us spend some words on the elements appearing in \eqref{pb}. \\
The matrix $A(t,x)$ is supposed to be bounded, coercive with only {measurable} coefficients. Then, the lack of regularity in the divergence operator prevents us to apply classical regularity estimates and we need to develop a suitable {nonlinear theory}. In particular, this means that nonlinear operators in divergence form are admitted as well.\\
The initial datum $u_0$ is supposed to be an {unbounded function} belonging to Lebesgue spaces and  the lack of boundedness implies that we cannot invoke maximum principles. \\
The $q$ power of the gradient makes such growth to be \emph{superlinear} (in some sense) but still subnatural $q<p$. To fix ideas, we assume that $q$ is strictly greater than a certain critical value $q_c$ which splits the interval $0<q<p$ between sublinear growths if $0<q\le q_c$ and superlinear ones $q_c<q<p$. \\
Finally, the coefficient $\ga$ is assumed to be strictly positive and then it gives a repulsive nature to the r.h.s.: roughly speaking,  the gradient term in the r.h.s. "fights against" the coercitivity of the l.h.s..\\


Let us give a brief overview on the literature behind problems of \eqref{pb} type.\\
As far as the case with {Laplace operator} in \eqref{pb} is concerned, regularizing effects and long time decays are dealt with in \cite{BDL,BASW,BD,PZ} regarding different notions of solutions (classical, mild and weak ones). In particular, when the initial datum is supposed to be continuous or simply bounded, decay estimates are proved when the gradient rate is positive $q>0$ with both repulsive and attractive nature (i.e. $\ga>0$ and $\ga<0$, respectively, in \eqref{pb}). In particular, \cite[Theorem $1.2$]{BDL} and \cite[Lemma $3.2$]{PZ} show that, if $1<q\le 2$, then the $L^\infty$-norms of both solutions and gradients decay to zero for large times with exponential rates: 
\[
\begin{array}{c}
\ds
\|u(t)\|_{L^\infty(\Omega)}\le K e^{-\lm_1 t},\\
[4mm]\ds
\|\N u(t)\|_{L^\infty(\Omega)}\le K (1+t^{-\frac{1}{2}})e^{-\lm_1 t},
\end{array}
\]
being $\lm_1$ the first eigenvalue of the Laplace operator with homogeneous Dirichlet boundary conditions. Note that this decay is sharp since it is satisfied also by the heat equation. We underline that the authors of \cite{BDL,PZ} can apply Bernstein's estimates, as well as linear semigroup theory or heat kernel estimates, which are not allowed in our general setting because of the assumptions on the matrix $A(t,x)$ in \eqref{pb}. 

\medskip

%
%

As already anticipated, our aim is dealing with {unbounded data in Lebesgue's spaces} 
\begin{equation}\label{d}
u_0\in L^\nu(\Omega)\q \t{for} \q \nu\ge 1,
\end{equation}
and thus, due to the presence of a superlinear term in \eqref{pb}, an explanation on the admissible values of $\nu$ is in order to be given. We underline that the need of taking care of the data regularity is due to the superlinear setting and does not depend on the nature of the superlinearity itself. 
For instance, we refer to \cite{BrCa} where the superlinearity has the form $|u|^q$, $q>1$ and to \cite{M} in our case.\\ 
As shown in \cite{BASW,BD} when $p=2$ and in \cite{M} for $1<p<N$ in a more general context, we need to fix 
$$
\nu\ge \max\left\{ 1,\frac{N(q-(p-1))}{p-q} \right\}
$$
in \eqref{d} in order to get an existence result when a superlinear growth in the gradient term occurs. The same compatibility condition was already observed in \cite{BASW} for the Cauchy problem with $p=2$. We remark that, when $q$ is superlinear, nonexistence counterexamples are proved if $1\le \nu<\frac{N(q-(p-1))}{p-q}$ in \cite[Subsection $3.2$]{BASW} for the Cauchy problem with Laplace operator in \eqref{pb} and in \cite[Section $7$]{MP} as far as the Cauchy-Dirichlet problem with $p=2$ in \eqref{pb} is concerned. 

\medskip

A {nonlinear approach}, aimed at studying the regularity and the behaviour in time of solutions of \eqref{pb} with $p=2$, is contained in \cite{MP}.
In particular, the main step relies on the proof of an a priori estimate for the level set function $G_k(u)=(|u|-k)_+\text{sign}(u)$ which has the form
\[
\begin{array}{c}
\ds
\sup_{t\in (0,T)}\|G_k(u(t))\|_{L^{\frac{N(q-1)}{2-q}}(\Omega)}^{\frac{N(q-1)}{2-q}}+\|\nabla\left[(1+|G_k(u)|)^{\frac{N(q-1)}{2(2-q)}}\right]\|_{L^2(Q_{T})}^2\le M\\
[4mm]\ds
\t{for}\q 2-\frac{N}{N+1}<q<2,
\end{array}
\]
where $k$ is taken large enough to have $\|G_k(u_0)\|_{L^{\frac{N(q-1)}{2-q}}(\Omega)}$ suitable small and with $M=M(\||u_0|\chi_{|u_0|>k}\|_{L^\nu(\Omega)})$.\\
Observing the inequality above, we deduce two important facts: first, we have that (morally) the function $G_k(u)$ acts like a subsolution of the coercive problem
\begin{equation}\label{pp}\tag{$\text{P}_{\text{c}}$}
\begin{cases}
\begin{array}{ll}
u_t-\dive a(t,x,u,\nabla u)=0 \qquad&\text{in}\,\,  Q_T,\\
u=0 \qquad &\text{on}\,\,(0,T)\times\partial\Omega,\\
u(0,x)=u_0(x)\qquad&\text{in}\,\,  \Omega,
\end{array}
\end{cases}
\end{equation} 
and so we expect that $G_k(u)$ inherits the own features of \eqref{pp}; moreover, looking at the energy term, we foresee that a well precise power $|u|^{\beta}$, $\beta=\beta(p,q,N)$, plays a certain role in the study of \eqref{pb}. \\
We are going to comment this last observation. Dealing with a general superlinear setting, then one has to require some regularity on the solutions in order to have the problem well posed. In this sense, we refer to \cite{BarPo,Po3} in the elliptic framework and \cite{MP,LeM} in the parabolic one. More precisely, a comparison result is proved in \cite[Section $6$]{MP} when the solution $u$ belongs to the regularity class
\[
\left\{u\,\,\text{solving \eqref{pb}}: \quad |u|^{\frac{N(q-1)}{2(2-q)}}\in  L^2(0,T;H_0^1(\Omega))
\right\}
\]
while nonuniqueness occurs (see \cite[Appendix A]{LeM}) if
\[
\left\{u\,\,\text{solving \eqref{pb}}: \quad |u|^\ro\in  L^2(0,T;H_0^1(\Omega))\quad\text{with}\,\,\ro<\frac{N(q-1)}{2(2-q)}
\right\}.
\]
See also \cite[Example $1.1$]{GMP} for an analogous observation in the elliptic framework.\\
In the same spirit, we quote \cite{ADAP} where \eqref{pb} is studied with $q=p=2$ and, due to the natural growth, the right class in which one has to study the problem is given by
\[
\left\{u\,\,\text{solving \eqref{pb} with } q=p=2: \quad (e^u-1)\in  L^2(0,T;H_0^1(\Omega))
\right\}.
\]


\medskip

We now recall some well known facts concerning coercive problems. 
Let us focus on \eqref{pp} for a while. We assume that $a(t,x,u,\xi):(0,T)\times\Omega\times \R\times\R^n\to \R^n$ verifies classical Leray-Lions structure conditions (see also \eqref{A}) and $u_0\in L^{\nu}(\Omega)$, $\nu \ge 1$.\\
We stress on the relation between the parameter $p$ and the Lebesgue summability $\nu$ of the initial datum.\\
If we consider values of $p$ that are smaller than the threshold $\frac{2N}{N+\nu}$, $\nu>1$, then  we \emph{cannot expect any regularizing effect} (see \cite[Theorem $1.2$]{Po3}).\\
On the contrary, as $\ds p>\frac{2N}{N+\nu}$ and $\nu\ge 1$, then \emph{a regularizing effect occurs}. Indeed, we have that (see \cite[Theorems $1.3$]{Po3})
\begin{equation}\label{S}
\|u(t)\|_{L^r(\Omega)}\le c\frac{\|u_0\|_{L^\nu(\Omega)}^{h_0}}{t^{h_1}}\qquad\text{a.e.} \,\,t\in (0,T),
\end{equation}
for $c=c(\al,r,p,\nu,N)$,
\begin{equation*}
h_0=\frac{\nu[2N-p(N+r)]}{r[2N-p(N+\nu)]}\qquad\text{and}\qquad h_1=\frac{N(\nu-r)}{r[2N-p(N+\nu)]}.
\end{equation*}
Furthermore, the case $r=\infty$ (\cite[Theorem $1.4$]{Po3}, \cite{Po2} and also \cite{CG} when $p=2$ and $\nu\ge 2$) is admitted and the decay estimate is given by
\begin{equation}\label{U}
\|u(t)\|_{L^\infty(\Omega)}\le c\frac{\|u_0\|_{L^\nu(\Omega)}^{\frac{p\nu}{p(N+\nu)-2N}}}{t^{\frac{N}{p(N+\nu)-2N}}}\qquad\text{a.e.} \,\,t\in (0,T),
\end{equation}
with $c=c(\al,p,\nu,N)$ and where the exponents follow from the limits
\[
\lim_{r\to \infty} h_0=\frac{p\nu}{p(N+\nu)-2N}\qquad\text{and}\qquad \lim_{r\to \infty}h_1=\frac{N}{p(N+\nu)-2N}.
\] 
Note that the above estimates, beyond the regularizing effect, can be read as decay estimates too. However, it is well known that \eqref{U} is not sharp in the sense that it can be refined with respect to great and small values of $t$ in bounded domains (see \cite{Di} and also the last part of \cite[Corollary $2.1$]{Po2} for $p>2$ and \cite{Fri} as $p=2$).\\
Finally, if either $\ds\frac{2N}{N+\nu}\le p<2$ and $\nu>1$ or $\ds\frac{2N}{N+1}< p<2$ and $\nu=1$, then \emph{extinction in finite time occurs} (see \cite[Theorems $1.5$ \&  $1.6$]{Po3}), i.e. there exists a time $\overline{T}$ such that 
\begin{equation}\label{E}
u(t,x)=0\qquad\forall t\ge \overline{T}.
\end{equation}

\section{Assumptions}\label{hp}

Let us present the problem we are going to study  in its generality.\\
We consider the following parabolic Cauchy-Dirichlet problem 
\begin{equation}\label{P} \tag{P}
\begin{cases}
\begin{array}{ll}
u_t-\dive a(t,x,u,\nabla u) =H(t,x,\nabla u) & \text{in}\,\,Q_T,\\ 
u=0  &\text{on}\,\,(0,T)\times\partial\Omega,\\
u(0,x)=u_0(x) &\text{in}\,\, \Omega,
\end{array}
\end{cases}
\end{equation}
assuming that the vectorial valued function $a(t,x,u,\xi):(0,T)\times\Omega\times \R\times\R^n\to \R^n$ satisfies classical Leray-Lions structure assumptions, namely
\begin{subequations}
\makeatletter
\def\@currentlabel{A}
\makeatother
\label{A} 
\renewcommand{\theequation}{A\arabic{equation}}
\begin{equation}\label{A1} 
\exists \alpha>0:\quad
\alpha|\xi|^p\le a(t,x,u,\xi)\cdot\xi,
\end{equation}
\begin{equation}\label{A2} 
\exists \lambda>0:\quad|a(t,x,u,\xi)|\le \lambda[|u|^{p-1}+|\xi|^{p-1}+h(t,x)]\quad\text{where }\,\, h\in L^{p'}(Q_T),
\end{equation}
\begin{equation}\label{A3} 
(a(t,x,u,\xi)-a(t,x,u,\eta))\cdot(\xi-\eta)> 0,
\end{equation}
\end{subequations}
for almost every $(t,x)\in Q_T$, for every $u\in \R$ and for every $\xi$, $\eta$ in $\R^N$ with $\xi\ne\eta$. \\
As far as the r.h.s. is concerned, we assume that it  grows at most as a power of the gradient
\begin{equation}\label{H0}\tag{H}
\exists  \gamma>0\,\,\text{s.t. }\quad
|H(t,x,\xi)|\le \gamma |\xi|^q 
\end{equation}
a.e. $(t,x)\in Q_T$, for all $\xi\in \mathbb{R}^N$, with superlinear $q$ rates belonging to the range
\begin{equation*}
\max\left\{\frac{p}{2},\frac{p(N+1)-N}{N+2} \right\}<q<p.
\end{equation*}
Note that this means that we are requiring $\ds q>\frac{p(N+1)-N}{N+2} $ if $p\ge 2$ and $\ds q>\frac{p}{2}$ as $p\le 2$.

We recall that the compatibility condition between the initial datum $u_0\in L^\nu(\Omega)$ and the $q$ growth of the gradient term is given by
\begin{equation}\label{compcond}
\nu=\max\left\{
1,\si\right\},\q \si=\frac{N(q-(p-1))}{p-q}.
\end{equation}
Then, if we have
\begin{equation}\label{Q1} \tag{$\text{Q}_\si$}
\max\left\{\frac{p}{2},p-\frac{N}{N+1} \right\}<q<p \quad\text{with} \quad 1<p<N
\end{equation}
in \eqref{H0}, we need to ask at least the following summability on the initial datum:
\begin{equation}\label{ID1} \tag{$\text{ID}_{\si}$}
u_0\in L^{\sigma}(\Omega)\quad\text{with}\quad \sigma=\frac{N(q-(p-1))}{p-q}.
\end{equation}
As the $q$ rate gets slower but keeps superlinear, i.e. 
\begin{equation}\label{Q2} \tag{$\text{Q}_1$}
\max\left\{\frac{p}{2},\frac{p(N+1)-N}{N+2} \right\}<q<p-\frac{N}{N+1}\quad\text{with}\quad \frac{2N}{N+1}<p<N,
\end{equation} 
we can consider $L^1(\Omega)$ data (see \eqref{compcond}):
\begin{equation}\label{ID2} \tag{$\text{ID}_1$}
u_0\in L^1(\Omega).
\end{equation}
We require  $\frac{2N}{N+1}<p$ in order to give sense to \eqref{Q2}. \\
The growth rates in \eqref{Q2} would allow us to deal even with measures data, since $\frac{N(q-(p-1))}{p-q}<1$. For further comments in this sense, we refer to \cite[Theorem $2.2$]{BASW}. However, we choose $L^1(\Omega)$ data in order to keep ourselves in the Lebesgue framework.

The particular case $q=p-\frac{N}{N+1}$ with $p>\frac{2N}{N+1}$ will be commented later with its own assumptions and, at this moment, we just observe that such a $q$ value is critical in the sense that it implies  that the value of $\si$ in \eqref{compcond} is exactly $1$. Note that such a $q$ growth represents the changing point between $L^\si(\Omega)$ and $L^1(\Omega)$ data.\\ 

Some words on the relation between the ranges of both $p$ and $q$, aimed at clarifying the data setting, are in order to be given. 
Let  us set
\begin{figure}[H]
\centering
\begin{tabular}{r|l}
\begin{tikzpicture}
\draw [very thick,dashed, color=red!70!black] (-1,0) to (2,0);
\draw [very thick, dashdotted, color=orange] (-1,-.2) to (2,-.2);
\end{tikzpicture}  
& $u_0\in L^\si(\Omega)$,
\\
\begin{tikzpicture}
\draw [very thick, color=yellow] (-1,0) to (2,0);
\end{tikzpicture}
& $u_0\in L^1(\Omega)$.
\end{tabular}
\captionsetup{labelformat=empty}
\caption{Colours legend}
\end{figure}
We sketch out our $q$ intervals on the real lines below with respect to the value of $p$, highlighting the cases $q=p-\frac{N}{N+1}$ and $q=p-\frac{N}{N+2}$ since they represent, respectively, the $L^1(\Omega)$ and the $L^2(\Omega)$ thresholds of the data (i.e. $\nu=\si=1$ if $q=p-\frac{N}{N+1}$ and $\nu=\si=2$ if $q=p-\frac{N}{N+2}$).\\
We have
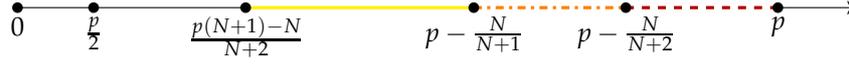
\begin{figure}[H]
\centering
\begin{tikzpicture}
\draw [thin] (0,0) to (3,0);
\draw [->,thin] (10,0) -- (11,0);
\draw [very thick, dashed, color=red!70!black] (8,0) to (10,0);
\draw [very thick, dashdotted,color=orange] (6,0) to (8,0);
\draw [very thick, color=yellow] (3,0) to (6,0);
\fill (0,0) circle (2pt) node[below] 
{$0$};
\fill (1,0) circle (2pt) node[below] {$\frac{p}{2}$};
\fill (3,0) circle (2pt) node[below] {$\frac{p(N+1)-N}{N+2}$};
\fill (6,0) circle (2pt) node[below] {$p-\frac{N}{N+1}$};
\fill (8,0) circle (2pt) node[below] {$p-\frac{N}{N+2}$};
\fill (10,0) circle (2pt) node[below] {$p$};
\end{tikzpicture}
\captionsetup{labelformat=empty}
\caption{The case $2\le p<N$}
\label{fig:1}
\end{figure}
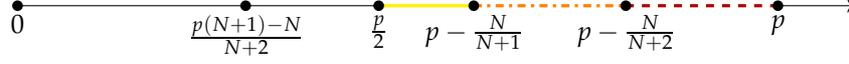
\begin{figure}[H]
\centering
\begin{tikzpicture}
\draw [thin] (0,0) to (4.75,0);
\draw [->,thin] (10,0) -- (11,0);
\draw [very thick,dashed,  color=red!70!black] (8,0) to (10,0);
\draw [very thick, dashdotted,color=orange] (6,0) to (8,0);
\draw [very thick, color=yellow] (4.75,0) to (6,0);
\fill (0,0) circle (2pt) node[below] 
{$0$};
\fill (3,0) circle (2pt) node[below] {$\frac{p(N+1)-N}{N+2}$};
\fill (4.75,0) circle (2pt) node[below] {$\frac{p}{2}$};
\fill (6,0) circle (2pt) node[below] {$p-\frac{N}{N+1}$};
\fill (8,0) circle (2pt) node[below] {$p-\frac{N}{N+2}$};
\fill (10,0) circle (2pt) node[below] {$p$};
\end{tikzpicture}
\captionsetup{labelformat=empty}
\caption{The case $\frac{2N}{N+1}<p< 2$}
\label{fig:2}
\end{figure}
As far as the cases $p-\frac{2N}{N+1}<\frac{p}{2}$ and $p-\frac{N}{N+2}<\frac{p}{2}$ are concerned, we have
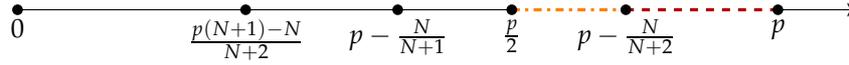
\begin{figure}[H]
\centering
\begin{tikzpicture}
\draw [thin] (0,0) to (6.5,0);
\draw [->,thin] (10,0) -- (11,0);
\draw [very thick,dashed,  color=red!70!black] (8,0) to (10,0);
\draw [very thick,dashdotted, color=orange] (6.5,0) to (8,0);
\fill (0,0) circle (2pt) node[below] 
{$0$};
\fill (3,0) circle (2pt) node[below] {$\frac{p(N+1)-N}{N+2}$};
\fill (5,0) circle (2pt) node[below] {$p-\frac{N}{N+1}$};
\fill (6.5,0) circle (2pt) node[below] {$\frac{p}{2}$};
\fill (8,0) circle (2pt) node[below] {$p-\frac{N}{N+2}$};
\fill (10,0) circle (2pt) node[below] {$p$};
\end{tikzpicture}
\captionsetup{labelformat=empty}
\caption{The case $\frac{2N}{N+2}<p\le \frac{2N}{N+1}$}
\label{fig:3}
\end{figure}
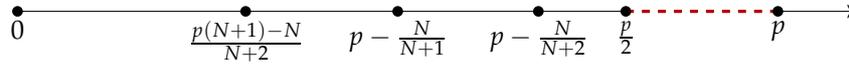
\begin{figure}[H]
\centering
\begin{tikzpicture}
\draw [thin] (0,0) to (8,0);
\draw [->,thin] (10,0) -- (11,0);
\draw [very thick, dashed, color=red!70!black] (8,0) to (10,0);
\fill (0,0) circle (2pt) node[below] 
{$0$};
\fill (3,0) circle (2pt) node[below] {$\frac{p(N+1)-N}{N+2}$};
\fill (5,0) circle (2pt) node[below] {$p-\frac{N}{N+1}$};
\fill (6.85,0) circle (2pt) node[below] {$p-\frac{N}{N+2}$};
\fill (8,0) circle (2pt) node[below] {$\frac{p}{2}$};
\fill (10,0) circle (2pt) node[below] {$p$};
\end{tikzpicture}
\captionsetup{labelformat=empty}
\caption{The case $\frac{2N}{N+\si}<p\le \frac{2N}{N+2}$}
\label{fig:4}
\end{figure}

Looking at the real lines above we deduce that
\begin{equation*}
q>\frac{p}{2}\quad\Leftrightarrow\quad p>\max\left\{\frac{2N}{N+\si},\frac{2N}{N+1}\right\}=\frac{2N}{N+\nu},\q \nu  \t{ in \eqref{compcond}},
\end{equation*}
which, roughly speaking, means that we have an existence result in the superlinear setting if and only if we have $p$ great enough. Note that the $p$ threshold $\frac{2N}{N+\nu}$ is the same as the coercive case \eqref{pp}. 
This means that we \emph{cannot} fall in the range $1<p\le \frac{2N}{N+\nu}$ if we want to keep the superlinear character of \eqref{P}.\\ 
We synthesise the above comments saying that \emph{if we are in the superlinear framework and a solution of \eqref{P}  exists, then such a solution regularizes}.


We collect in the figure below our incoming decay results.
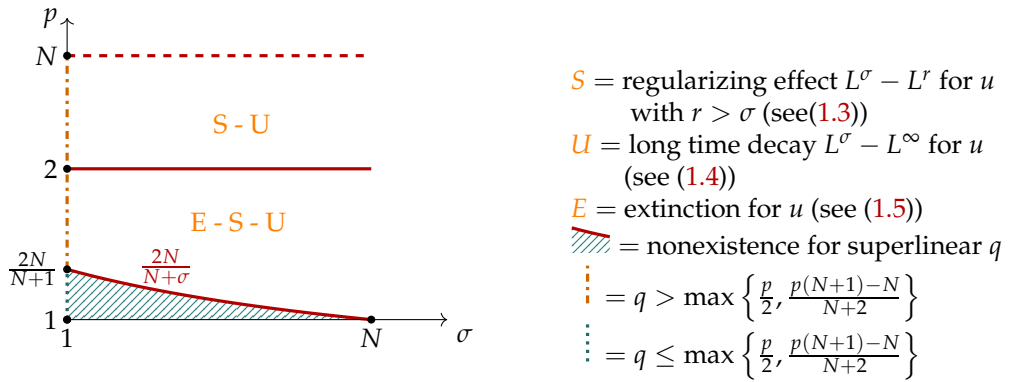
\begin{center}
	\begin{tabular}{ll}
		\begin{minipage}{9.5cm}
			\begin{figure}[H]
				\begin{center}
					\begin{tikzpicture}
					\fill [domain=1:5, pattern=
					north east lines
					, pattern color=air,
					variable=\x](1,1)
					-- plot ({\x}, {10/(5+\x)})
					-- (5,1)
					-- cycle;
					\draw[domain=1:5,smooth,very thick,variable=\x,red!70!black] (1,1) plot ({\x},{10/(5+\x)});
					\draw[very thick, dashed, red!70!black] (1,4.5) -- (5,4.5);
					\draw[very thick, red!70!black] (1,3) -- (5,3);
					\node[right] at (2.8,3.6) {\textcolor{orange}{S - U}};
					\node[right] at (2.5,2.3) {\textcolor{orange}{E - S - U}};
					\node at (2.3,1.7) {\textcolor{red!70!black}{$\frac{2N}{N+\si}$}};
					\draw[->] (1,1) -- (6,1) node[anchor=north west] {$\si$};
					\draw[->] (1,4.5) -- (1,5) node[left] {$p$};
					\draw[very thick,dashdotted,orange!80!black] (1,5/3) -- (1,4.5);
					\draw[very thick, dotted, air!80!black] (1,1) -- (1,5/3);
					\fill (5,1) circle (1.5pt) node[below] {$N$};
					\fill (1,3) circle (1.5pt) node[left] {$2$};
					\fill (1,4.5) circle (1.5pt) node[left] {$N$};
					\fill (1,5/3) circle (1.5pt) node[left] {$\frac{2N}{N+1}$};
					\fill (1,1) circle (1.5pt) node[below] {$1$};
					\fill (1,1) circle (1.5pt) node[left] {$1$};
					\end{tikzpicture}
					\begin{center}
\captionsetup{labelformat=empty}
						\caption{Regularizing effect estimates and long time decays w.r.t. $p$ and $\si$}
					\end{center}
				\end{center}
			\end{figure}
		\end{minipage}
		&
\hspace*{-2cm}
		\begin{minipage}{8cm}
			\begin{figure}[H]
				\begin{center}
					\parbox[b]{\textwidth}{
						\begin{tabbing}
							${\color{orange} S}=$ regularizing effect $L^{\si}-L^r$ for $u$\\ $\qquad$ with $r>\si$ (see\eqref{S})\\
							${\color{orange} U}=$ long time decay $L^{\si}-L^{\infty}$ for $u$ \\  $\qquad$(see \eqref{U})\\
							${\color{orange} E}=$ extinction for $u$ (see \eqref{E})\\
							\begin{tikzpicture}[scale=.5]
							\fill [domain=1:2, pattern=
							north east lines
							, pattern color=air,
							variable=\x](1,1)
							-- plot ({\x}, {10/(5+\x)})
							-- (2,1)
							-- cycle;
							\draw[domain=1:2,smooth,very thick,variable=\x,red!70!black] (1,1) plot ({\x},{10/(5+\x)});
							\end{tikzpicture}
							$=$ nonexistence for superlinear $q$\\
							\,\,
							\begin{tikzpicture}[scale=.5]
							\draw[very thick,dashdotted,orange!80!black] (1,1) -- (1,2);
							\end{tikzpicture}
							\,\,$=$ $ q>\max\left\{\frac{p}{2},\frac{p(N+1)-N}{N+2}\right\}$\\
							\,\,
							\begin{tikzpicture}[scale=.5]
							\draw[very thick,dotted,air!80!black] (1,1) -- (1,2);
							\end{tikzpicture}
							\,\,$=$ $ q\le\max\left\{\frac{p}{2},\frac{p(N+1)-N}{N+2}\right\}$\\
						\end{tabbing}
					}
				\end{center}
			\end{figure}
			
		\end{minipage}
	\end{tabular}
\end{center}

We point out that obtaining decays results in superlinear settings is not obvious: for instance, solutions of the superlinear power problem 
\[
u_t-\D u=|u|^q\qq \t{with} \q q>1
\]
may blow up in finite time (see \cite{BrCaYvRa,P2}).

\subsection*{Notation} We will represent by $c,\,C$ positive constants which may vary from line to line, specifying also its dependence on the parameters. We name $c_S$, $c_P$ and $c_{GN}$, respectively, the Sobolev embedding constant, the Poincar\'{e} constant and the  constant due to the Gagliardo-Nirenberg inequality. We also define the functions $G_k(z)$ and $T_k(z)$ as
\[
G_k(z)=(|z|-k)_+\text{sign}(z)\quad\text{and}\quad T_k(z)=\max\{-k, \min\{k, v\}\}\q\forall k>0 .
\]
Note that, from the above definitions, one has 
\begin{equation}\label{dec}
z=G_k(z)+T_k(z).
\end{equation}

\section{The growth range with $L^\si(\Omega)$ data}\label{Lsi}

\setcounter{equation}{0}
\renewcommand{\theequation}{\thesection.\arabic{equation}}

\numberwithin{equation}{section}

This Section is devoted to the growth case \eqref{Q1} which, we recall,  requires Lebesgue data satisfying at least \eqref{ID1}. \\
We point out that we could split the range \eqref{Q1}  into two main parts with respect to the value of $\si$. Indeed, problem \eqref{P}  admits solutions with finite energy (see \cite[Theorem $4.5$]{M}) if either
\be\begin{array}{c}\label{fe}
\ds
p-\frac{N}{N+2}\le q<p\quad\text{and}\quad \frac{2N}{N+2}< p <N\\[3mm]\ds
\text{or}\\[2mm]\ds
\frac{p}{2}<q<p\quad\text{and}\quad \frac{2N}{N+\si}<p\le \frac{2N}{N+2}
\end{array}\end{equation}
are in force, since such $q$ growths imply that \eqref{ID1} satisfies, respectively, 
 $\si\ge 2$ and $\si >2$. As $q$ gets smaller, so does the value of $\si$ and finite energy solutions are not allowed any more. In particular, when we consider
\begin{equation*}
\max\left\{ \frac{p}{2},p-\frac{N}{N+1}   \right\}<q<p-\frac{N}{N+2} 
\end{equation*}
then we have $1<\si<2$ in \eqref{ID1}.\\
With the aim to deal with the range \eqref{Q1}  at once, we here introduce a notion of solution which is inspired by the renormalized setting.
We first define $\mathcal{T}^{1,p}_0(Q_T)$ as the set of all measurable functions $u:Q_T\to \R$ almost everywhere finite and such that the truncated functions $T_k(u)$ belong to $ L^p(0,T;W^{1,p}_0(\Omega))$ for all $k>0$:
\begin{equation*}
\mathcal{T}^{1,p}_0(Q_T)=\left\{u:Q_T\to \R\quad\text{a.e. finite}\,\,:\,\, T_k(u)\in  L^p(0,T;W^{1,p}_0(\Omega))\quad \forall k>0  \right\}.
\end{equation*}

\begin{definition}\label{defrin1}
We say that a function $u\in {\mathcal{T}^{1,p}_0(Q_T)}$ is a solution of \eqref{P}  if satisfies \eqref{pot} and 
\begin{equation*}
H(t,x,\nabla u)\in L^1(Q_T), 
\end{equation*}
\begin{equation}\label{sr2}
\begin{array}{c}
\ds
-\integrale S(u_0)\vp (0,x)\,dx+\iint_{Q_T}-S(u)\vp_t+a(t,x,u,\N u)\cdot \N (S'(u)\vp)\,dx\,ds\\
[4mm]\ds
=\iint_{Q_T}H(t,x,\N u)S'(u)\vp\,dx\,ds
\end{array}
\end{equation}
for every $S\in W^{2,\infty}(\R)$ such that $S'(\cdot)$ has compact support and for every test function $\vp\in C_c^\infty([0,T)\times \Omega)$ such that $ S'(u)\vp\in L^p(0,T;W^{1,p}_0(\Omega))$ (i.e. $S'(u)\vp$ is equal to zero on $(0,T)\times\partial\Omega$).
\end{definition}

\medskip

Roughly speaking, the notion of renormalized solution moves the attention from the solution $u$ to its truncated function $T_k(u)$, which has now finite energy. For further comments on this notion of solution we refer to \cite{BDGM,BlMu,DMOP,Mu}.  We also underline that, unlike the above references do, we do not require any asymptotic condition on the energy term such as 
\begin{equation*}
	\lim_{n\to \infty}\frac{1}{n}\iint_{\{n\le |u|\le 2n\}}a(t,x,u,\N u)\cdot \N u=0,
\end{equation*}
since it is implied by the regularity class we are going to consider (see \eqref{pot} below).\\
Let us introduce our \emph{regularity class}:
\begin{equation}\label{pot}\tag{RC}
(1+|u|)^{\beta-1}u\in L^p(0,T;W^{1,p}_0(\Omega)),\quad\beta=\frac{\si+p-2}{p}.
\end{equation}
The existence of solutions of \eqref{pb} has been proved in \cite[Theorems $4.5$ \& $5.4$]{M}.
We underline that dealing with solutions which enjoy \eqref{pot} is crucial since it determines the \emph{well posedness class} of \eqref{P}.
We note also that, if $\si\ge 2$ (i.e.  \eqref{fe} hold), then $\beta\ge 1$ and so \eqref{pot} provides us with a stronger information than only knowing $u\in L^p(0,T;W^{1,p}_0(\Omega))$.\\ 

In order to deal with our current framework, we here define the function $\te_n(\cdot)$ as below:
\begin{center}
	\begin{tabular} {ll} 
\hspace*{-.5cm}
		\begin{minipage}{70mm} 
			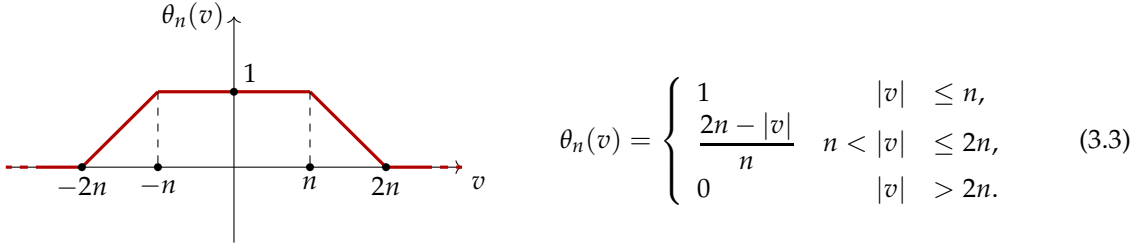
\begin{figure}[H]
				\centering
				\begin{tikzpicture}
				\draw[->] (0,0) -- (3,0) node[anchor=north west] {$v$};
				\draw[->] (0,0) -- (0,2) node[left] {$\theta_n(v)$};
				\draw[-] (-3,0) -- (0,0);
				\draw[-] (0,-1) -- (0,0);
				\draw[very thick, red!70!black] (-1,1) -- (1,1);
				\draw[very thick, red!70!black] (1,1) -- (2,0);
				\draw[very thick, red!70!black] (-1,1) -- (-2,0);
				\draw[very thick, red!70!black] (2,0) -- (2.5,0);
				\draw[very thick, red!70!black] (-2,0) -- (-2.5,0);
				\draw[very thick, dashed, red!70!black] (2.5,0) -- (3,0);
				\draw[very thick, dashed, red!70!black] (-2.5,0) -- (-3,0);
				\draw[dashed] (1,1) -- (1,0);
				\draw[dashed] (-1,1) -- (-1,0);
				\fill (1,0) circle (1.5pt) node[below] {$n$};
				\fill (-1,0) circle (1.5pt) node[below] {$-n$};
				\fill (0,1) circle (1.5pt) node[above right] {$1$};
				\fill (2,0) circle (1.5pt) node[below] {$2n$};
				\fill (-2,0) circle (1.5pt) node[below] {$-2n$};
				\end{tikzpicture}
\captionsetup{labelformat=empty}
				\caption{The function $\te_n(v)$}
			\end{figure}
		\end{minipage} 
		\begin{minipage}{80mm} 
			\hypertarget{suppcmpt}
			{\begin{equation}\label{suppcmpt}
				\theta_n(v)=
				\begin{cases} 
				\begin{array}{lrl}
				\ds
				1& |v|&\ds \le n,\\
				\ds
				\frac{2n-|v|}{n} & n<|v|&\ds \le 2n,\\
				\ds
				0 &|v|&\ds >2n.
				\end{array}
				\end{cases}
				\end{equation}}
		\end{minipage} 
	\end{tabular} 
\end{center}
\medskip
Note that $\theta_n(v)$ is compactly supported and converging to $1$. 

\subsection{$L^{\si}(\Omega)-L^{\si}(\Omega)$ regularity }

\renewcommand{\theequation}{\thesection.\arabic{equation}}

\numberwithin{equation}{section}

Our first result contains the key point of our next ones and we will refer to this particular step as \emph{the $\de$ argument}. Roughly speaking, we prove that a contraction in the $L^\si$-norm, $\si>1$ as in \eqref{ID1}, holds for the level set function $G_k(u(t))$ provided that this is initially ($t=0$) not too big (i.e., $k$ is large). We underline that, when dealing with the $G_k(\cdot)$ function, \emph{no smallness conditions} on the initial datum are assumed, but eventually it is enough to take a large $k$.  An analogous $\de$ argument has already been used in \cite{MP} where \eqref{P}  is studied under the assumptions in \cref{hp} when $p=2$.  


\medskip

\begin{lemma}\label{lemGkp} 
Assume \eqref{ID1},  \eqref{A1}--\eqref{A2} with $p>\frac{2N}{N+\si}$  and
\eqref{H0} with  \eqref{Q1}. Moreover, let $u$ be a solution of \eqref{P}  in the sense of \cref{defrin1}. Then, there exists a positive value $\delta_0$ such that, for every $k>0$ and for every $\delta<\delta_0$ satisfying
\[
\|G_k(u_0)\|_{L^{\si}(\Omega)}^{\si}<\delta,
\]
we have
\begin{equation*}
\|G_k(u(t))\|_{L^{\sigma}(\Omega)}^{\si}<\delta\quad \forall  t\in[0,T].
\end{equation*}
\end{lemma}
\begin{proof}
We claim that the function $S'(\cdot)=S_{n,\eps}'(\cdot)$ defined as
\[
S_{n,\eps}'({G_k(u)})\vp=\int_0^{{T_n}(G_k(u))}(\eps+|v|)^{\si-3}|v|\,dv\quad\text{with} \quad \vp=1
\]
can be taken in \eqref{sr2}. Indeed, even if it is not compactly supported, the regularity assumption \eqref{pot} allows us to proceed by standard arguments for renormalized solutions (i.e., beginning with ${\te_h(G_k(u))}\,S_{n,\eps}'(G_k(u))\vp$ where $\te_h(\cdot)$ is defined in \eqref{suppcmpt}, recalling \eqref{A1}  and \eqref{H0} and then letting $h\to\infty$). Then, thanks also to the growth assumption in \eqref{H0}, we get
\[
\begin{array}{c}
\ds
\int_{\Omega} S_{n,\eps}(G_k(u(t)))\,dx+{\al}\int_0^t \|\nabla \Phi_\varepsilon (T_n(G_k(u(s))))\|_{L^p(\Omega)}^p\,ds\\
[4mm]\ds
\le \int_{\Omega} S_{n,\eps}(G_k(u_0))\,dx+\gamma  \iint_{Q_t} |\nabla G_k(u)|^q\biggl(\int_0^{G_k(u)}(\varepsilon+|z|)^{\sigma-3}|z|\,dz\biggr)\,dx\,ds,
\end{array} 
\]
where $\Phi_\varepsilon(v)=\int_0^v (\varepsilon+|z|)^{\frac{\sigma-3}{p}}|z|^{\frac{1}{p}}\,dz$.  The definition of $\Phi_\varepsilon(\cdot)$ allows us to estimate the second term in the above r.h.s. as 
\begin{equation*}\begin{array}{c}
\ds
\gamma  \iint_{Q_t} |\nabla G_k(u)|^q\biggl(\int_0^{G_k(u)}(\varepsilon+|z|)^{\sigma-3}|z|\,dz\biggr)\,dx\,ds\\
[3mm]\ds
\le\ga
\iint_{Q_t}|\N \Phi_\varepsilon(G_k(u))|^q\biggl(\int_0^{G_k(u)}(\varepsilon+|z|)^{(\sigma-3)\frac{p-q}{p}}|z|^{\frac{p-q}{p}}\,dz\biggr)\,dx\,ds\\
[3mm]\ds
\le\gamma  \iint_{Q_t} |\nabla \Phi_\varepsilon(G_k(u))|^q 
|\Phi_\varepsilon(G_k(u))|^{p-q}|G_k(u)|^{q-p+1}\,dx\,ds,
\end{array}\end{equation*}
where the last step is due to H\"older's inequality with indices $\left(\frac{1}{p-q},\frac{1}{q-(p-1)}\right)$ (we recall that $q>\max\left\{ \frac{p}{2},p-\frac{N}{N+1} \right\} >p-1$). An application of the H\"older inequality with $\left(\frac{p}{q},\frac{p^*}{p-q},\frac{N}{p-q}\right)$, Sobolev's embedding and the definition of $\si$ (we just recall here that $\si=
{N(q-(p-1))}/{(p-q)}$) give us
\begin{equation*}
\begin{array}{c}
\ds
\int_{\Omega} S_{n,\eps}(G_k(u(t)))\,dx+{\al}\int_0^t \|\nabla \Phi_\varepsilon (T_n(G_k(u(s))))\|_{L^p(\Omega)}^p\,ds\\
[4mm]\ds
\le \int_{\Omega} S_{n,\eps}(G_k(u_0))\,dx+\ga c_S\left(\sup_{s\in (0,t)}\|G_k(u(s))\|_{L^{\sigma}(\Omega)}^\si\right)^{\frac{p-q}{N}}\int_0^t \|\nabla \Phi_\varepsilon (G_k(u(s)))\|_{L^p(\Omega)}^p\,ds.
\end{array}
\end{equation*}
Being $\frac{p-q}{N}<1$ and thanks to \eqref{ID1}  and \eqref{Q1}, we deduce that $\int_{\Omega} S_{n,\eps}(G_k(u(t)))\,dx<\infty$ uniformly in $n$ and for fixed $\eps$. In particular,  we gain the boundedness of $\|G_k(u)\|_{L^\infty(0,T;L^\si(\Omega))}$.\\
 Such a result, combined with \eqref{pot} and \eqref{ID1}, allows us to consider the limit for $n\to \infty$ in the previous inequality getting 
\begin{equation}\label{in}
\begin{array}{c}
\ds
\int_{\Omega} S_{\eps}(G_k(u(t)))\,dx+{\al}\int_0^t \|\nabla \Phi_\varepsilon (G_k(u(s)))\|_{L^p(\Omega)}^p\,ds\\
[4mm]\ds
\le \int_{\Omega} S_{\eps}(G_k(u_0))\,dx+\ga c_S\left(\sup_{s\in (0,t)}\|G_k(u(s))\|_{L^{\sigma}(\Omega)}^\si\right)^{\frac{p-q}{N}}\int_0^t \|\nabla \Phi_\varepsilon (G_k(u(s)))\|_{L^p(\Omega)}^p\,ds,
\end{array}
\end{equation}
where $S_{\eps}(x)=\int_0^{x}\left(\int_0^{y}(\eps+|z|)^{\si-3}|z|\,dz\right)\,dy$. In particular, thanks again to \eqref{pot}, we deduce the convergence to zero of $\int_{\Omega} S_{\eps}(G_k(u(t)))\,dx$ for $k\to\infty$ which provides us the one of $\int_{\Omega}|G_k(u(t))|^{\si}\,dx\to 0$  for  $k\to \infty$. \\ 
Then, the continuity regularity $u\in C([0,T];L^{\si}(\Omega))$ follows combining this last convergence with \cite[Theorem $1.1$]{P1} (which implies that $u\in C([0,T];L^{1}(\Omega))$), the decomposition in \eqref{dec} and by an application of the Vitali Theorem. \\

\noindent
{\it The $\de$ argument.}\\
Let us focus on \eqref{in}. 
We choose a value $\de_0$ such that $\ds 0<
\ga c_S\de_0^{\frac{p-q}{N}}<\al$ and a value $ k_0$ large enough so that
\begin{equation}\label{u0reg}
\|G_k(u_0)\|_{L^{\sigma}(\Omega)}^{\sigma}<\de\quad\forall k\ge k_0
\end{equation} 
for fixed $\de<\de_0$.\\
Moreover, always considering $k\ge k_0$, we set
\begin{equation*}
T^*:=\sup\{s\in [0,T]: \,\|G_k(u(t))\|_{L^{\sigma}(\Omega)}^{\sigma}\le\de \q\forall  t\le s  \}
\end{equation*}
and we have that $T^*>0$ due to the continuity regularity just proved and to \eqref{u0reg}.\\
Choosing $t\le T^*$ in \eqref{in} and recalling the definition of $\de$, we manage to absorb the r.h.s. obtaining
\begin{equation}\label{dis}
\int_{\Omega} S_{\eps}(G_k(u(t)))\,dx+\left(\al-\ga c_S\de^{\frac{p-q}{N}}\right)\iint_{Q_t}|\N \Phi_\varepsilon(G_k(u))|^p\,dx\,ds
\le \int_{\Omega} S_{\eps}(G_k(u_0))\,dx.
\end{equation}
Moreover, since the convergence $S_{\eps}(G_k(u(s)))\underset{\vare\to 0}{\longrightarrow}\frac{|G_k(u(s))|^\sigma}{\sigma(\sigma-1)}$ holds, \eqref{dis} provides us with the contraction
\begin{equation}\label{datotp}
\int_{\Omega}|G_k(u(t))|^{\sigma}\,dx\le \int_{\Omega}|G_k(u_0)|^{\sigma}\,dx\q\forall k\ge k_0.
\end{equation}
The inequality \eqref{datotp} can be extended to the whole interval $[0,T]$ reasoning by contradiction. Let us suppose that $T^*<T$.
Then, the definition of $T^*$ and  \eqref{u0reg} lead to
\[
\de=\int_{\Omega}|G_k(u(T^*))|^{\sigma}\,dx\le \int_{\Omega}|G_k(u_0)|^{\sigma}\,dx<\de\q\forall k\ge k_0
\]
which is in contrast with the definition of $T^*$ because of continuity $u\in C([0,T];L^{\si}(\Omega))$.

\end{proof}

We here state an important consequence which derives from the $\de$ argument above.

\medskip

\begin{corollary}\label{corinftyp} 
Assume  $u_0\in L^{\infty}(\Omega)$, \eqref{A1}--\eqref{A2} and
\eqref{H0} with  \eqref{Q1}. Moreover, let $u$ be a solution of \eqref{P}  in the sense of \cref{defrin1}. Then, we have that $u\in L^{\infty}(Q_T)$. Moreover, the following contraction estimate holds:
\[
\|u(t)\|_{L^{\infty}(\Omega)}\le \|u_0\|_{L^{\infty}(\Omega)}\quad \forall  t\in[0,T].
\]
\end{corollary}
\begin{proof}
	The assertion can be easily deduced taking $k_0=\|u_0\|_{L^\infty(\Omega)}$ in \eqref{datotp}.
\end{proof}

Roughly speaking, this contraction result implies that if one manages to prove that $u$ is bounded at a certain time $\tau$, then it keeps bounded and the $L^\infty$-norm decreases in the time variable. 

\medskip
\begin{lemma}\label{propW11p}
Assume \eqref{ID1},  \eqref{A1}--\eqref{A2} with $p>\frac{2N}{N+\si}$ and
\eqref{H0} with  \eqref{Q1}. Moreover, let $u$ be a solution of \eqref{P} in the sense of \cref{defrin1} and consider $\Phi:\mathbb{R}\to \mathbb{R}$ be a $C^2$ convex function such that
\begin{equation}\label{ga}
\Phi'(0)=0\quad\text{and}\quad \Phi''(\xi)\le c(1+|\xi|)^{\si-2},
\end{equation}
for some constant $c>0$. Then the function $t\to\int_{\Omega}\Phi(u(t))\,dx$ belongs to $W^{1,1}(0,T)$ and satisfies 
\begin{equation}\label{point}
\frac{\text{d}}{\text{d}t}\int_{\Omega}\Phi(u(t))\,dx+\int_{\Omega}a(t,x,u,\nabla u)\cdot \nabla u\, \Phi''(u)\,dx
=\int_\Omega H(t,x,\nabla u)\,\Phi'(u)\,dx
\end{equation}
a.e. in $t\in(0,T)$.
\end{lemma}
\begin{proof}
We omit the proof since it is very similar to the one proposed in \cite[Lemma $3.1$]{MP}. We just observe that the growth assumption \eqref{ga} plays the role of \eqref{pot}. In particular, \eqref{ga} is needed to justify the choice of $S_n'(\cdot)\vp=\Phi_n'(\cdot)$, $\vp=1$, in \eqref{sr2} (i.e., we begin with $S_n'(z)\vp={\te_h(z)}\Phi_n'(z)$, where $\te_h(\cdot)$ has been defined in \eqref{suppcmpt}; then, thanks to \eqref{A1}, \eqref{H0},  we let $h\to\infty$).

\end{proof}
Here we propose the generalization of \cite[Proposition $3.2$]{MP} in which the $L^\si(\Omega)-L^\si(\Omega)$ long time decay of \eqref{P} is proved with $p=2$ in \eqref{A}, \eqref{ID1} and \eqref{Q1}.

\medskip

\begin{proposition}\label{ex} 
Assume \eqref{ID1},  \eqref{A1}--\eqref{A2} with $p>\frac{2N}{N+\si}$ and 
\eqref{H0} with  \eqref{Q1}. Moreover, let $u$ be a solution of \eqref{P} in the sense of \cref{defrin1}. Then, for $k$ sufficiently large (say $k\ge k_0$ with $k_0$ as in \cref{lemGkp}), we have that
\begin{equation}\label{secp}
\frac{\text{d}}{\text{d}t}\int_{\Omega}|G_k(u(t))|^{\si}\,dx+\frac{\si}{\b^p}\left(\al-\ga c_S\de^{\frac{p-q}{N}}\right)
\int_{\Omega}  |\nabla [|G_k(u)|^{\beta}] |^p  \,dx
\le 0
\end{equation}
a.e. $t\in(0,T)$, for all $k\ge k_0$ (see \eqref{dis}).\\
Furthermore, for $\lm=\frac{c_S\,\si}{\b^p}\left(\al-\ga c_S\de^{\frac{p-q}{N}}\right)|\Omega|^{-\frac{N(p-2)+p\si}{N\si}}$ and $k\ge k_0$, we have that 
\begin{itemize}
\item  if $2<p<N$, then $\|G_k(u(t))\|_{L^{\si}(\Omega)}$ decreases in the time variable and the following polynomial decay holds:
\[
\|G_k(u(t))\|_{L^{\si}(\Omega)}\le \biggl( \|G_k(u_0)\|_{L^{\si}(\Omega)}^{-(p-2)}
+\lm\frac{p-2}{\si} t\biggr)^{-\frac{1}{p-2}}\q\forall  t\ge 0;
\] 
\item if $\frac{2N}{N+\si}< p<2$, there exists a positive time $\overline{T}$ such that
\[
G_k(u)=0\qquad\forall t\ge \overline{T}.
\]
In particular, such a value $\overline{T}$ is given by 
\[
\overline{T}=\frac{\si}{(2-p)\lm}\|G_k(u_0)\|_{L^{\si}(\Omega)}^{2-p}.
\]
\end{itemize}
\end{proposition}
\begin{proof}
The inequality in \eqref{secp} follows combining \cref{propW11p} with \cref{lemGkp}. Indeed, invoking \cref{propW11p} with $\Phi_\varepsilon'(G_k(u))=(\si-1)\int_0^{G_k(u)}(\eps+w)^{\si-2}\,dw$ and reasoning as in \cref{lemGkp} (see \eqref{dis}), we obtain
\[
\frac{\text{d}}{\text{d}t}\int_{\Omega} \Phi_\eps({G_k(u(t))})\,dx+\left(\al-\ga c_S\de^{\frac{p-q}{N}}\right)\int_{\Omega}|\N \Phi_\varepsilon(G_k(u))|^p\,dx 
\le 0
\]
and \eqref{secp} is recovered once we let $\eps$ vanish.\\

We go further observing that, by definitions of $\si$ and $\beta$, we have that 
\[
\si< \beta p^*\,\, \Leftrightarrow\,\, p> \frac{2N}{N+\si}
\]
and thus, thanks to Sobolev's embedding and to Lebesgue's spaces inclusion, we can estimate from below as follows:
\begin{align}
0&\ge\frac{\text{d}}{\text{d}t}\|G_k(u(t))\|_{L^{\si}(\Omega)}^{\si}+\frac{\si}{\b^p}\left(\al-\ga c_S\de^{\frac{p-q}{N}}\right)\|\nabla [|G_k(u(t))|^{\beta}]\|_{L^p(\Omega)}^p\notag\\
&\ge \frac{\text{d}}{\text{d}t}\|G_k(u(t))\|_{L^{\si}(\Omega)}^{\si}+\lm\|G_k(u(t))\|_{L^{\si}(\Omega)}^{\beta p}\label{betap}
\end{align}
for every $k\ge k_0$ and where 
$\lm=\frac{c_S}{\b^p}\left(\al-\ga c_S\de^{\frac{p-q}{N}}\right)|\Omega|^{-\frac{N(p-2)+p\si}{N\si}}$. We  set
\[
y(s)=\|G_k(u(s))\|_{L^{\si}(\Omega)}^{\si}
\]
and rewrite \eqref{betap} as
\begin{equation}\label{dis1}
y'(s)+\lm y(s)^{\frac{\beta p}{\si}}\le 0\qquad \forall k\ge k_0.
\end{equation}
We now split the rest of the proof with respect to the cases $p>2$ and $p<2$.

Let $\ds 2<p<N$: in this way, we have that $\frac{\beta p}{\si}=\frac{\si+p-2}{\si}>1$ and then Gronwall's type Lemma (see, e.g., \cite[Lemma $3.1$]{Po2}) provides us with
\[
y(t)\le \biggl(y(0)^{-\frac{p-2}{\si}}+\lambda\frac{p-2}{\si}t\biggr)^{-\frac{\si}{p-2}}\qquad\forall t\in (0,T), \,\, \forall k\ge k_0.
\]
%


Having $\ds\frac{2N}{N+\si}< p<2$ guarantees that $\frac{\beta p}{\si}<1$ and \eqref{dis1} gives us  
\[
y(t)^{\frac{2-p}{\si}}\le y(0)^{\frac{2-p}{\si}}-\lm\frac{2-p}{\si}t\qquad\forall t\in (0,T), \,\, \forall k\ge k_0
\]
from which we deduce that $y(t)=0$ if $t\ge \overline{T}=\frac{\si}{\lm(2-p)}y(0)^{\frac{2-p}{\si}}$. 

The assertions follow recalling the definitions of $y(\cdot)$ and $\lm$.
\end{proof}

\subsection{The regularizing effect $L^{\si}(\Omega)-L^{r}(\Omega)$}

\renewcommand{\theequation}{\thesection.\arabic{equation}}

\numberwithin{equation}{section}

\begin{proposition}\label{Prop2rp}
Assume \eqref{ID1},  \eqref{A1}--\eqref{A2} with $p>\frac{2N}{N+\si}$ and 
\eqref{H0} with  \eqref{Q1}  and let $u$ be a solution of \eqref{P} in the sense of \cref{defrin1}. Then
\begin{equation}\label{CLr}
u\in C((0, T);L^{r}(\Omega))\q \t{for}\q r>\si.
\end{equation}
Moreover, there exists a value $k_0$, independent of $r$, such that the regularizing effect can be expressed through the decay estimate
\begin{equation}\label{starstar}
\|G_k(u(t))\|_{L^{r}(\Omega)}^{r}\le c\frac{\|G_k(u_0)\|_{L^{\si}(\Omega)}^{\si\frac{N(p-2)+pr}{N(p-2)+p\si}}}{t^{\frac{N(r-\si)}{N(p-2)+p\si}}}\qquad \forall t\in(0,T),\,\,\forall k\ge k_0,
\end{equation}
where $c=c(\gamma,r,q,p,\alpha,N)$. Furthermore we have the short time decay
\begin{equation}\label{std}
\|u(t)\|_{L^{r}(\Omega)}^{r}\le \frac{C}{t^{\frac{N(r-\si)}{N(p-2)+p\si}}}\qquad \forall t\in(0,t_0]
\end{equation}
where $C=C(\gamma,r,q,p,\alpha,N,t_0, u_0,|\Omega|)$.
\end{proposition}

\begin{proof} 
We set $\Phi(\cdot)= S(\cdot)$ in \eqref{point}, with $S\in W^{2,\infty}(\mathbb{R})$ satisfying
\begin{equation}\label{S1}
0\le S''(v)\le c(\eps+|v|)^{p(\beta-1)-1}|v|=c(\eps+|v|)^{\si-3}|v| 
\end{equation}
and
\begin{equation}\label{S2}
\frac{S'(v)}{\left(S''(v)\right)^\frac{q}{p}}\le L\int_0^{v} \left(S''(y)\right)^\frac{p-q}{p}\,dy,
\end{equation}
for some positive constants $c,\,L$. Again, we justify such a choice of $S(\cdot)$ reasoning as in \cref{lemGkp} and taking advantage of \eqref{S1}, since this last condition plays the same role of \eqref{pot}.\\ 
Then, letting $S(\cdot)=S(G_k(u(t)))$ and recalling \eqref{H0}, we have that the following differential inequality
\begin{equation}\label{w11rin}
\begin{array}{c}
\ds
\frac{\text{d}}{\text{d}t}\int_{\Omega}S({G_k(u(t))})\,dx+\al\int_{\Omega}|\N G_k(u)|^p S''(G_k(u))\,dx
\le\ga\int_\Omega |\N G_k(u)|^qS'(G_k(u))\,dx 
\end{array}
\end{equation}
holds a.e. $t\in(0,T)$.

We now define $\Psi(G_k(u))=\int_0^{G_k(u)} \left(S''(y)\right)^\frac{1}{p}\,dy$ and use \eqref{S2} in \eqref{w11rin}, obtaining
\begin{equation*} 
\begin{array}{c}
\ds
\frac{\text{d}}{\text{d}t}\int_{\Omega}S({G_k(u(t))})\,dx+\al\int_{\Omega}|\N \Psi(G_k(u))|^p\,dx
\\
[4mm]\ds
\le \ga L\int_\Omega |\N G_k(u)|^q(S''(G_k(u)))^{\frac{q}{p}} \left(\int_0^{G_k(u)} \left(S''(z)\right)^\frac{p-q}{p}\,dz\right)\,dx 
\end{array}
\end{equation*}
from which, being
\[
\int_0^{G_k(u)} \left(S''(v)\right)^\frac{p-q}{p}\,dv\le \left(\int_0^{G_k(u)} \left(S''(v)\right)^\frac{1}{p}\,dv\right)^{p-q}|G_k(u)|^{q-(p-1)}
\]
by H\"older's inequality with $\left( \frac{1}{p-q},\frac{1}{q-(p-1)} \right)$, we get 
\begin{equation*}
\begin{array}{c}
\ds
\frac{\text{d}}{\text{d}t}\int_{\Omega}S(G_k(u(t)))\,dx+\al \integrale |\N \Psi(G_k(u))|^p\,dx\\
[3mm]\ds
\le  \ga L\integrale |\N \Psi(G_k(u))|^q\left(\Psi(G_k(u))\right)^{p-q}|G_k(u)|^{q-(p-1)}\,dx.
\end{array}
\end{equation*}
Another application of H\"older's inequality with indices $\left(\frac{p}{q},\frac{p^*}{p-q},\frac{N}{p-q}\right)$ and Sobolev's embedding give us
\[
\scalebox{.96}[1]{$\ds\frac{\text{d}}{\text{d}t}\int_{\Omega}S(G_k(u(t)))\,dx+\al \integrale |\N \Psi(G_k(u))|^p\,dx\le L\ga c_S\sup_{t\in (0,T)}\|G_k(u(t))\|_{L^{\si}(\Omega)}^{q-(p-1)}\integrale |\N \Psi(G_k(u))|^p\,dx$}
\]
and then, invoking \cref{lemGkp} with $k_0$ sufficiently large in order to have $\al>L\ga c_S\de ^{\frac{p-q}{N}}$, we finally get  
\begin{equation}\label{phi}
\frac{\text{d}}{\text{d}t}\int_{\Omega}S(G_k(u(t)))\,dx+c_1 \integrale |\N \Psi(G_k(u))|^p\,dx\le 0\q\forall k\ge k_0
\end{equation}
where $c_1=\al-L\ga c_S\de ^{\frac{p-q}{N}}$.\\
We now fix a value $r>\si$ and define 
\begin{equation}\label{test1}
S'(v)=S_{n,\eps}'(v)=\int_0^v (\eps+|y|)^{\si-3}|y|T_n(y)^{r-\si}\,dy\quad\text{if}\quad 1<\si<2,
\end{equation}
\begin{equation}\label{test2}
S'(v)=S_n'(v)=\int_0^v |T_n(y)|^{r-2}\,dy\quad\text{if}\quad \si\ge 2.
\end{equation}
Note that, for fixed $n$, \eqref{test1}--\eqref{test2} are admissible choices of $S'(\cdot)$ since they verify both \eqref{S1} and \eqref{S2}.\\
Our current goal is characterising the relation between  
$$
S_{n,\eps}(G_k(u))\q\t{and}\q \Psi(G_k(u))=\Psi_{n,\eps}(G_k(u))=\int_0^{G_k(u)} \left(S''_{n,\eps}(y)\right)^\frac{1}{p}\,dy\quad\text{when}\quad 1<\si<2,
$$
$$ 
S_n(G_k(u))\q\t{and}\q\Psi(G_k(u))=\Psi_n(G_k(u))=\int_0^{G_k(u)} \left(S''_{n}(y)\right)^\frac{1}{p}\,dy\quad\text{when}\quad \si\ge 2,
$$
 in order to rewrite \eqref{phi} only in terms of $S_{n,\eps}(G_k(u))$ and $S_n(G_k(u))$. To this aim, we split the rest of the proof with respect to the value of $\si$.\\
Let us consider the case $1<\si<2$ first. 
We start with an estimate of the test function \eqref{test1} itself. Let $\om\in (0,1)$ to be fixed later. Then, by H\"older's inequality with $\left(\frac{1}{p^*\om},1-\frac{1}{p^*\om}\right)$, we get
\begin{equation*}
\begin{array}{c}
\ds
\int_0^y(\eps+|z|)^{\si-3}|z|T_n(z)^{r-\si}\,dz\\
[4mm]\ds
\le \left( \int_0^y\left((\eps+|z|)^{\si-3}|z|T_n(z)^{r-\si} \right)^\frac{1}{p}\,dz\right)^{\om p^*}
\left( \int_0^y\left((\eps+|z|)^{\si-3}|z|T_n(z)^{r-\si} \right)^\frac{N-pN\om}{N-p-Np\om}\,dz\right)^{1-\om p^*}.
\end{array}
\end{equation*}
Since it holds that $(\eps+|z|)^{\si-3}|z|T_n(z)^{r-\si}\le |z|^{r-2}$ being $\si<2$, we improve the inequality above as
\begin{equation*}
\begin{array}{c}
\ds
\int_0^y(\eps+|z|)^{\si-3}|z|T_n(z)^{r-\si}\,dz\\
[4mm]\ds
\le \left( \int_0^y\left((\eps+|z|)^{\si-3}|z|T_n(z)^{r-\si} \right)^\frac{1}{p}\,dz\right)^{\om p^*}
|y|^{r-1-\frac{N\om}{N-p}(r-2+p)}\\
[4mm]\ds
=|\Psi_{n,\eps}(y)|^{p^*\omega}|y|^{r-1-\frac{N\om}{N-p}(r-2+p)}
.
\end{array}
\end{equation*}
Finally, we fix $\om=\frac{(r-\si)(N-p)}{N(r-\si+p-2)+p\si}$ in order to have $r-1-\frac{N\om}{N-p}(r-2+p)=\si(1-\om)-1$ and conclude saying that the previous steps and the definition of $S_{n,\eps}(\cdot)$ in \eqref{test1} lead us to 
\begin{align*}
S_{n,\eps}(G_k(u(s)))
&\le c(r)|\Psi_{n,\eps}(G_k(u(s)))|^{p^*\omega}\int_0^{G_k(u(s))}|y|^{\si(1-\omega)-1} \,dy\\
&\le c(r)|\Psi_{n,\eps}(G_k(u(s)))|^{p^*\omega} |G_k(u(s))|^{\si(1-\omega)},
\end{align*}
so we get
\begin{equation}\label{fin1}
\integrale S_{n,\eps}(G_k(u(s)))\le c(r) \left(\integrale |\Psi_{n,\eps}(G_k(u(s)))|^{p^*}\,dx\right)^\om \|G_k(u_0)\|_{L^\si(\Omega)}^{\si(1-\om)}
\end{equation}
as desired.
The inequality in \eqref{fin1} implies that \eqref{phi}, read in terms of $S_{n,\eps}(\cdot)$ and $\Psi_{n,\eps}(\cdot)$, can be estimated from below as
\begin{equation}\label{disfingkp1} 
\ds\frac{\text{d}}{\text{d}s}\int_{\Omega} S_{n,\eps}(G_k(u(s)))\,dx+c_2\left[
\frac{\int_{\Omega}S_{n,\eps}(G_k(u(s)))\,dx}{\|G_k(u_0)\|_{L^{\si}(\Omega)}^{\si(1-\omega)}}\right]^{\frac{p}{p^*\omega}}\le 0
\end{equation}
a.e. $s\in (0,T]$, for all $k\ge k_0$ and with $c_2$ depending on $\al,\,\ga,\,N,\,q,\,p$ and $r$.\\
We  integrate the inequality in \eqref{disfingkp1} between $0<s\le t$, getting
\[
\int_{\Omega}\int_0^{G_k(u(t))}\left(\int_0^v(\eps+|z|)^{\si-3}|z|T_n(z)^{r-\si}\,dz\right)\,dv\,dx\le \frac{c_2}{t^{\frac{N(r-\si)}{N(p-2)+p\si}}}
\|G_k(u_0)\|_{L^{\si}(\Omega)}^{\frac{p\si(1-\omega)}{p^*\omega}\frac{N(r-\si)}{N(p-2)+p\si}}
\]
Note that $\frac{N(r-\si)}{N(p-2)+p\si}>0$ since $\frac{2N}{N+\si}<p$ and $r>\si$.\\
We finally apply the Fatou Lemma on $n$ and on $\eps$ in the previous inequality so that, recalling the definition of $\omega$, we obtain
\begin{equation}\label{gkr}
\|G_k(u(t))\|_{L^{r}(\Omega)}^{r}\le c_2\frac{\|G_k(u_0)\|_{L^{\si}(\Omega)}^{\si\frac{N(p-2)+pr}{N(p-2)+p\si}}}{t^{\frac{N(r-\si)}{N(p-2)+p\si}}}\qquad \text{a.e.}\,\,t\in(0,T),\,\,\forall k\ge k_0.
\end{equation}

\medskip

We now deal with the case $\si\ge 2$. 
We rewrite $r=p^*\frac{r-2+p}{p}\omega+\si(1-\omega)$ where, as in the previous case, $\om=\frac{(r-\si)(N-p)}{N(r-\si+p-2)+p\si}\in (0,1)$ . An application of Holder's inequality with $(\frac{1}{\omega},\frac{1}{1-\omega})$, combined with the inequality
\[
\int_0^{y}|T_n(z)|^{p^*\frac{r-2+p}{p}-1}\,dz\le c(r)\left(\int_0^{y}|T_n(z)|^{\frac{r-2+p}{p}-1}\,dz\right)^{p^*},
\]
gives us
\begin{equation*}
\begin{split}
\int_0^{y} |T_n(z)|^{r-2}\,dz
&\le \biggl(\int_0^{y}|T_n(z)|^{p^*\frac{r-2+p}{p}-1}\,dz\biggr)^{\omega} \biggl(\int_0^{y}|T_n(z)|^{\si-1-\frac{1}{1-\omega}}\,dz\biggr)^{1-\omega}\\
&\le c(r)\biggl(\int_0^{y} |T_n(z)|^{\frac{r-2}{p}}\,dz\biggr)^{\omega p^*}|y|^{\si(1-\omega)-1}\\
&\le c(r)|\Psi_n(y)|^{\omega p^*}|y|^{\si(1-\omega)-1},
\end{split}
\end{equation*}
from which, recalling \eqref{test2}, we deduce 
\begin{align*}
S_n(G_k(u(s)))&
=\int_0^{G_k(u(s))}\left(\int_0^y |T_n(z)|^{r-2}\,dz\right)\,dy\\
&\le c(r)|\Psi_n(G_k(u(s)))|^{p^*\omega}\int_0^{G_k(u(s))}|y|^{\si(1-\omega)-1} \,dy\\
&\le c(r) |\Psi_n(G_k(u(s)))|^{p^*\omega} |G_k(u(s))|^{\si(1-\omega)}.
\end{align*}
This step, together with Holder's inequality with $(\frac{1}{\omega},\frac{1}{1-\omega})$ and the monotonicity of $\|G_k(u(s))\|_{L^\si(\Omega)}$ for large values of $k$ (see \cref{ex}), implies that
\begin{align*}
\int_{\Omega}S_n(G_k(u(s)))\,dx&\le c(r)\int_{\Omega}\biggl(|\Psi_n(G_k(u(s)))|^{p^*\omega}|G_k(u(s))|^{\si(1-\omega)}\biggr)\,dx\\
&\le c(r) \biggl(\int_{\Omega}|\Psi_n(G_k(u(s)))|^{p^*}\,dx\biggr)^{\omega}\|G_k(u(s))\|_{L^{\si}(\Omega)}^{\si(1-\omega)}\\
&\le c(r) \biggl(\int_{\Omega}|\Psi_n(G_k(u(s)))|^{p^*}\,dx\biggr)^{\omega}\|G_k(u_0)\|_{L^{\si}(\Omega)}^{\si(1-\omega)}
\end{align*}
and we get again an estimate from below for $\int_{\Omega}|\Psi_n(G_ku)|^{p^*}\,dx$ in terms of $\int_{\Omega} S_n(G_k(u(s)))\,dx$. 
Then  \eqref{phi}, read in terms of $S_n(\cdot)$ and $\Psi_{n}(\cdot)$, can be estimated from below as
\begin{equation*}
\ds\frac{\text{d}}{\text{d}s}\int_{\Omega} S_n(G_k(u(s)))\,dx+c_3\left[
\frac{\int_{\Omega}S_n(G_k(u(s)))\,dx}{\|G_k(u_0)\|_{L^{\si}(\Omega)}^{\si(1-\omega)}}\right]^{\frac{p}{p^*\omega}}\le 0
\end{equation*}
a.e. $s\in (0,T]$, for all $k\ge k_0$ and with $c_3$ depending on $\al,\,\ga,\,N,\,q,\,p$ and $r$, thanks also to Sobolev's embedding.\\
The inequality in \eqref{gkr}, with a possibly different constant depending on $\al,\,L,\,\ga,\,N,\,q,\,p,\,k_0$ and $r$, follows reasoning as before.\\

The decomposition \eqref{dec} implies that we  also have
\[
\begin{split}
\|u(t)\|_{L^{r}(\Omega)}^{r}&\le \|G_{k_0}(u(t))\|_{L^{r}(\Omega)}^{r}+k_0^r|\Omega|\\
&\le c\,\frac{\|G_{k_0}(u_0)\|_{L^{\si}(\Omega)}^{\si\frac{N(p-2)+pr}{N(p-2)+p\si}}}{t^{\frac{N(r-\si)}{N(p-2)+p\si}}}+k_0^r|\Omega|\\
&\le c\,\frac{\de^{\si\frac{N(p-2)+pr}{N(p-2)+p\si}}}{t^{\frac{N(r-\si)}{N(p-2)+p\si}}}+k_0^r|\Omega|
\end{split}
\]
for a.e. $t\in(0,T)$, $c$ depending on $\al,\,\ga,\,N,\,q,\,p$ and $r$ and where $\de$ is a constant depending on the equi-integrability of $u_0$ in $L^\si(\Omega)$ (see \cref{lemGkp}). Then, we deduce that the decay
\[
\|u(t)\|_{L^{r}(\Omega)}^{r}\le C \frac{\de^{\si\frac{N(p-2)+pr}{N(p-2)+p\si}}}{t^{\frac{N(r-\si)}{N(p-2)+p\si}}}\qquad \text{a.e.}\,\,t\in(0,t_0),\,\,\forall k\ge k_0
\]
holds for small times and with positive constant $C=C(\gamma,r,q,p,\alpha,N,t_0,u_0,|\Omega|)$.\\

The continuity regularity in \eqref{CLr} follows invoking the Vitali's Theorem and so do \eqref{starstar}--\eqref{std}. 
\end{proof}

\medskip

\begin{remark}\label{rmk2}
We claim that the previous \cref{Prop2rp} implies that 
\begin{equation}\label{regr}
u\in L^{\infty}((\tau,T];L^{r}(\Omega))\cap L^{r-2+p}(\tau,T;L^{p^*\frac{r-2+p}{p}}(\Omega))\q\t{for}\,\,\tau>0
\end{equation}
since
\begin{itemize}
\item  the regularity $u\in L^{\infty}((\tau,T];L^{r}(\Omega))$ directly follows from \eqref{CLr};
\item the regularity $u\in L^{r-2+p}(\tau,T;L^{p^*\frac{r-2+p}{p}}(\Omega))$ is due to the definitions of $S'(G_k(u))$ in \eqref{test1}--\eqref{test2}, since both implies that
\[
\left(1+|G_k(u)|\right)^{\frac{r-2+p}{p}}\in L^p(0,T;W^{1,p}(\Omega)).
\]
In particular, considering also the limit in $n\to \infty$ and in $\eps\to 0$ in \eqref{phi}, we have that
\begin{equation}\label{phin}
\frac{\text{d}}{\text{d}t}\int_{\Omega}|G_k(u(t))|^{r}\,dx+C_1
\integrale |\nabla \left(1+|G_k(u(t))|\right)^{\frac{r+p-2}{p}}|^p \,dx\,ds\le 0 \q\forall k\ge k_0,
\end{equation}
where $C_1=C_1(\al,\ga,N,p,q,r)$.
\end{itemize}
This fact implies that the function $t\to\int_{\Omega}|u(t)|^r\,dx$ belongs to $W^{1,1}(0,T)$ since, once we know \eqref{regr}, then we can reason as in \cref{propW11p}. 
\end{remark}

%
\subsection{Long  time decay results}\label{asfinen}

\renewcommand{\theequation}{\thesection.\arabic{equation}}

\numberwithin{equation}{section}

So far, the generalization of \cite{MP} to the case $p\ne 2$ strictly follows the methods adopted in this work.  However, once we get interested in the $L^\infty$-regularity, we change approach. More precisely, in \cite{MP} it is shown that the analogies between \eqref{P} (when $p=2$) and superlinear power problems (see, for instance, \cite{P2}) can be exploited to reason through a Moser type iteration argument, gaining the boundedness of solutions for positive times.  The general case $p\ne 2$ could be reasonably dealt with a similar argument. However, we choose to apply the results contained in \cite{Po2}.

\medskip

\begin{proposition}\label{rmk1}
Assume \eqref{ID1},  \eqref{A1}--\eqref{A2} with $p>\frac{2N}{N+\si}$ and  
\eqref{H0} with  \eqref{Q1}  and let $u$ be a solution of \eqref{P} in the sense of \cref{defrin1}. Then the function $G_k(u)$ satisfies the decays of the coercive problem \eqref{pp} for $k$ suitable large, i.e. 
\begin{equation}\label{decgkp}
\begin{array}{c}
\ds
\|G_k(u(t))\|_{L^{\infty}(\Omega)}\le c\frac{\|G_k(u_0)\|_{L^{\si}(\Omega)}^{h_0}}{t^{h_1}}\qquad \forall  k\ge k_0,\,\,\forall t>0,
\\
[3mm]\ds
\text{with}\quad h_1=\frac{N}{N(p-2)+p\si},\quad h_0=h_1\frac{p\si}{N}=\frac{p\si}{N(p-2)+p\si}
\end{array}
\end{equation}
and where $c$ is a constant depending on $N$, $q$, $p$, $\al$, $\ga$ and on some fixed value $r>\si$. Furthermore, if $p>2$, we have the following universal bound:
\begin{equation}\label{unifbou}
\|G_k(u(t))\|_{L^{\infty}(\Omega)}\le \frac{\overline{C}}{t^{\frac{1}{p-2}}}\qquad  \forall  k\ge k_0,\,\,\forall t>0,
\end{equation}
where $\overline{C}$ is a positive constant depending on $\al$, $\ga$, $N$, $p$, $q$, $|\Omega|$, $u_0$ and on $r$.
\end{proposition}
\begin{proof}

Consider the differential inequality \eqref{phin} in \cref{rmk2} and integrate between $\tau<s<t$. Then, thanks to Sobolev's inequality, we have
\begin{equation*}
\begin{split}
\int_{\Omega}|G_k(u(t))|^{r}\,dx-\int_{\Omega}|G_k(u(\tau))|^{r}\,dx+C_1 c_S
\int_{\tau}^t\left(\integrale |G_k(u)|^{\frac{r+p-2}{p}p^*} \,dx\right)^{\frac{p}{p^*}}\,ds
\le 0  ,
\end{split}
\end{equation*}
where $ k\ge k_0$, $k_0$ has been fixed as in \cref{Prop2rp} and $C_1$ is the same constant appearing in \eqref{phin} (we recall that $C_1$ depends on $\al$, $\ga$, $N$, $p$, $q$ and $r$).\\
The inequality above still holds if we consider $G_h(G_k(u))$ instead of ${G_k(u)}$. We point out that $h$ is an arbitrary positive fixed value but we always need to take $k\ge k_0$ as in \cref{Prop2rp}.

 So far, we already know that 
\[
u\in C([0,T];L^{\si}(\Omega))\cap C((\tau,T];L^{r}(\Omega))\cap L^{r+p-2}(\tau,T;L^{p^*\frac{r+p-2}{p}}(\Omega)),
\]
\[
\si<r<p^*\frac{r+p-2}{p}, 
\]
\[ \frac{r-\si}{1-\frac{\si}{p^*\frac{r+p-2}{p}}}<r+p-2<p^*\frac{r+p-2}{p}\quad\text{being}\quad \frac{2N}{N+\si}< p,
\]
\medskip
\begin{equation*}\label{cond1}
\|G_h(G_k(u(t)))\|_{L^{\si}(\Omega)}\le  \| G_k(u_0)\|_{L^{\si}(\Omega)}\qquad \forall  h>0,\, \forall  k\ge k_0,
\end{equation*}
\medskip
\begin{equation}\label{cond2}
\begin{array}{c}
\ds\integrale |G_h(G_k(u(t)))|^{r}\,dx-\integrale |G_h( G_k(u(\tau)))|^{r}\,dx\\
[4mm]\ds+C_1 c_S\int_\tau^t \|G_h(G_k(u(s)))\|_{L^{p^*\frac{r+p-2}{p}}(\Omega)}^{r+p-2}\,ds\le 0\qq\forall h>0,\,\forall k\ge k_0.
\end{array}
\end{equation}
Since the constant $C_1$ in \eqref{cond2} does not depend on $h$, we apply \cite[Theorem $2.1$]{Po2} to $G_k(u)$ and deduce \eqref{decgkp}.\\

If $p>2$, then $r<r+p-2$ and thus we can invoke again \cite[Theorem $2.2$]{Po2} gaining the universal bound in \eqref{unifbou} where  $\overline{C}$ is a positive constant depending on $\al$, $\ga$, $N$, $p$, $q$, $r$ and $|\Omega|$.

We point out that $\frac{1}{p-2}$ does not depend on the summability of the initial datum. Moreover, being $\frac{1}{p-2}>h_1$ then \eqref{unifbou} gives a stronger decay than \eqref{decgkp} for great values of $t$. Summarizing, we can say that if $p>2$, then
\[
\|G_k(u(t))\|_{L^{\infty}(\Omega)}\le c\min\left\{ \frac{\|G_k(u(\tau))\|_{L^{\si}(\Omega)}^{h_0}}{t^{h_1}},\frac{1}{t^{\frac{1}{p-2}}}\right\}\qquad  \forall t\in (0,T), \,\, \forall  k\ge k_0.
\]


\end{proof}
As a consequence of the decay above and \eqref{dec}, we gain the  boundedness for positive times of the solution $u$.

\medskip
So far, we have that $G_k(u)$ \emph{behaves as solutions of the coercive problem \eqref{pp} if $k$ is large enough}. This is not surprising since $G_k(u)$ satisfies a differential inequality of the type \eqref{secp}, of course for great value of $k$. \\
The next Proposition provides us with the long time decay of the $L^\infty$-norm of the whole solution. 

\medskip


\begin{proposition}\label{decinfp}
Assume \eqref{ID1},  \eqref{A1}--\eqref{A2} with $p>\frac{2N}{N+\si}$ and 
\eqref{H0} with  \eqref{Q1}. Moreover, let $u$ be a solution of \eqref{P}  in the sense of \cref{defrin1}. Then, we have
\[
\lim_{t\to\infty}\|u(t)\|_{L^{\infty}(\Omega)}=0.
\]
\end{proposition}
\begin{proof} 
We skip the proof of this result and say that, once we have \cref{lemGkp} and the decay in \eqref{decgkp}, then it can be proved as in \cite[Proposition $3.10$]{MP}.


\end{proof}

\medskip

\begin{remark}[A new smallness condition]\label{tau}
We claim that the results proved so far for $G_k(u)$ hold for the whole solution $u$ as well, up to consider large values of $t$. Indeed, by \cref{decinfp}, it is now sufficient to replace the smallness of $\|G_k(u(t))\|_{L^{\si}(\Omega)}$ (for great $k$) with the one of $\|u(t)\|_{L^{\infty}(\Omega)}$ (for large $t$) and then taking $k_0=0$ in \cref{lemGkp}. 
\end{remark} 

\medskip

\begin{proposition}\label{decsi}
Assume \eqref{ID1},  \eqref{A1}--\eqref{A2} with $p>\frac{2N}{N+\si}$ and 
\eqref{H0} with  \eqref{Q1}. Moreover, let $u$ be a solution of \eqref{P}  in the sense of \cref{defrin1}. Then, if $\tau$ is sufficiently large such that $\al-\ga c_S\|u(\tau)\|_{L^\infty(\Omega)}^{\frac{p-q}{N}}>0$ and for  $\lm=\frac{c_S\,\si}{\b^p}\left(\al-\ga c_S\|u(\tau)\|_{L^\infty(\Omega)}^{\frac{p-q}{N}}\right)|\Omega|^{-\frac{N(p-2)+p\si}{N\si}}$, we have that
\begin{itemize}
\item if $\ds 2<p<N$, then $\|u(t)\|_{L^{\si}(\Omega)}$ is decreasing in the time variable for $t>\tau$ and the following polynomial decay holds:
\[
\|u(t)\|_{L^{\si}(\Omega)}\le \biggl( \|u(\tau)\|_{L^{\si}(\Omega)}^{-(p-2)}
+\lm\frac{p-2}{\si} (t-\tau)\biggr)^{-\frac{1}{p-2}}\qq
  \forall  t\ge \tau;
\] 
\item if $\ds \frac{2N}{N+\si}< p<2$, then there exists a positive time {$\overline{T}$} such that
\[
u=0\qquad\forall t\ge \overline{T}+\tau.
\]
\end{itemize}
In particular, we can consider 
\[
\overline{T}=\frac{\si}{\lm(2-p)}\,\|u(\tau)\|_{L^\infty(\Omega)}^{\frac{2-p}{\si}}.
\]
\end{proposition}
\begin{proof}
We omit the proof since it is very similar to the one of \cref{ex}, up to replacing the smallness condition in \cref{lemGkp} with the one proposed in \cref{tau}.
\end{proof}

\medskip

\begin{proposition} 
Assume \eqref{ID1},  \eqref{A1}--\eqref{A2} with $p>\frac{2N}{N+\si}$ and  
\eqref{H0} with  \eqref{Q1}. Moreover, let $u$ be a solution of \eqref{P}  in the sense of \cref{defrin1}. Then
$$
u\in C((0, T);L^{r}(\Omega))\q\t{for}\q r>\si.
$$
Furthermore, there exists a value $ \tau$ such that the regularizing effect can be expressed through the decay estimate
\[
\|u(t)\|_{L^{r}(\Omega)}^{r}\le c\frac{\|u(\tau)\|_{L^{\si}(\Omega)}^{\si\frac{N(p-2)+pr}{N(p-2)+p\si}}}{(t-\tau)^{\frac{N(r-\si)}{N(p-2)+p\si}}}\qquad \forall t>\tau,
\]
where $c=c(\gamma,r,q,p,\alpha,N,|\Omega|)$.
\end{proposition}

\begin{proof}
We omit the proof since, thanks to \cref{tau},  it is very similar to the one of \cref{Prop2rp}, up to replacing the smallness condition in \cref{lemGkp} with the one proposed in \cref{tau}.
\end{proof}


\medskip

\begin{theorem}\label{decinf}
Assume \eqref{A1}--\eqref{A2} , \eqref{ID1}  and
\eqref{H0} with \eqref{Q1} . Moreover, let $u$ be a solution of \eqref{P}  in the sense of \cref{defrin1}. Then, the following polynomial decays hold for $2<p<N$
\begin{equation}\label{decup}
\|u(t)\|_{L^{\infty}(\Omega)}\le 
\begin{array}{ll}
\begin{cases}
&\ds C\,\|u(\tau)\|_{L^{\si}(\Omega)}^{\frac{p\si}{N(p-2)+p\si}}{(t-\tau)^{-\frac{N}{N(p-2)+p\si}}}\\
&\ds C_\tau(t-\tau)^{-\frac{1}{p-2}}
\end{cases}
\end{array}
\qquad \,\,\forall \,t>\tau,
\end{equation}
where $h_0$, $h_1$ are defined in \eqref{decgkp}, $C$ is a positive constant depending on $q$, $p$, $N$, $r$, $\al$, $|\Omega|$ and $u_0$ whether $C_\tau$ depends also on $\tau$.
\end{theorem}

\medskip

Even if this result immediately follows by \cref{rmk1} and \cref{tau}, we presents a short guideline which puts in evidence the use of \cref{decinfp}.

\begin{proof}

We verify that the assumptions of \cite[Theorem $2.1$ \& $2.2$]{Po2} hold.\\
We already know that
\[
u\in C([0,T];L^{\si}(\Omega))\cap C((\tau,T];L^{r}(\Omega))\cap L^{r+p-2}(\tau,T;L^{p^*\frac{r+p-2}{p}}(\Omega))
\]
with
\[
\si<r<p^*\frac{r+p-2}{p}, \qquad \frac{r-\si}{1-\frac{\si}{p^*\frac{r+p-2}{p}}}<r+p-2<p^*\frac{r+p-2}{p}
\]
and that
\begin{equation*}
\|{G_k(u(t))}\|_{L^{\si}(\Omega)}\le  \| G_k(u(\tau))\|_{L^{\si}(\Omega)}\qquad \, \forall  k\ge 0
\end{equation*}
are satisfied thanks to \cref{lemGkp}.\\
We are left with the proof of 
\begin{equation}\label{disc2}
\integrale |G_k(u(t))|^{r}\,dx-\integrale | G_k(u(\tau))|^{r}\,dx+\bar{c}\int_\tau^t \|G_k(u(s))\|_{L^{p^*\frac{r+p-2}{p}}(\Omega)}^{r+p-2}\,ds\le 0 
\end{equation}
for all $k\ge 0$ and where constant $\bar{c}$ does not depend neither on $k$ nor on the solution.\\
To this aim, we choose $|G_k(u)|^{r-2}G_k(u)$, $r> \si$,  as test function. Then, 
we have
\begin{equation*}
\begin{array}{c}
\ds
\frac{1}{r}\frac{\text{d}}{\text{d}s}\integrale |G_k(u(s))|^{r}\,dx+\alpha (r-1)\integrale |\nabla G_k(u(s))|^p|G_k(u(s))|^{r-2}\,dx\\
[3mm]\ds
\le \ga\integrale |\nabla G_k(u(s))|^q|G_k(u(s))|^{r-1}\,dx.
\end{array}
\end{equation*}
We apply H\"older's inequality with indices $\left(\frac{p}{q},\frac{p}{p-q}\right)$ in the r.h.s. obtaining
\[
\begin{array}{c}
\ds
\integrale |\nabla G_k(u(s))|^q|G_k(u(s))|^{r-1}\,dx=\integrale |\nabla G_k(u(s))|^q|G_k(u(s))|^{\frac{q}{p}(r-2)}|G_k(u(s))|^{r-1-\frac{q}{p}(r-2)}\,dx\\
[4mm]\ds
\le \left(
\integrale |\nabla G_k(u(s))|^p|G_k(u(s))|^{r-2}\,dx
\right)^{\frac{q}{p}}
\left(
\integrale |G_k(u(s))|^{\frac{p}{p-q}\left(
r-1-\frac{q}{p}(r-2)
\right)}\,dx
\right)^{\frac{p-q}{p}}.
\end{array}
\]
Then, since the equality $\frac{p}{p-q}\left(r-1-\frac{q}{p}(r-2)\right)=p\frac{r+p-2}{p}+\frac{p\si}{N}$ holds by definition of $\si$, the $L^{\infty}((\tau,T)\times \Omega)$ regularity and then the Poincar\'e inequality give us
\[
\begin{array}{c}
\ds
\left(
\integrale |\nabla G_k(u(s))|^p|G_k(u(s))|^{r-2}\,dx
\right)^{\frac{q}{p}}
\left(
\integrale |G_k(u(s))|^{
p\frac{r+p-2}{p}+\frac{p\si}{N}
}\,dx
\right)^{\frac{p-q}{p}}\\
[4mm]\ds
\le 
\|u(s)\|_{L^\infty(\Omega)}^{q-(p-1)}
\left(
\integrale |\nabla G_k(u(s))|^p|G_k(u(s))|^{r-2}\,dx
\right)^{\frac{q}{p}}
\left(
\integrale |G_k(u(s))|^{
p\frac{r+p-2}{p} 
}\,dx
\right)^{\frac{p-q}{p}}\\
[4mm]\ds
\le 
c_P\|u(s)\|_{L^\infty(\Omega)}^{q-(p-1)}
\integrale |\nabla G_k(u(s))|^p|G_k(u(s))|^{r-2}\,dx\,.
\end{array}
\]
Thus, for $\tau$ large enough such that $\tilde{c}=\al (r-1)-\ga c_P\|u(\tau)\|_{L^{\infty}(\Omega )}^{q-p+1}>0$, we have
\[
\frac{\text{d}}{\text{d}s}\integrale |G_k(u(s))|^{r}\,dx+\frac{\tilde{c}\,r\,p^p}{(r+p-2)^p}\integrale |\nabla[|G_k(u(s))|^{\frac{r+p-2}{p}}]|^p\,dx\le 0 \qquad\forall k\ge0.
\]
Having $r>\si$, an integration in the time variable for $\tau<s\le t$ provides us with \eqref{disc2} with $\bar{c}=\frac{\tilde{c}\,r\,p^p}{(r+p-2)^p} c_S$ in \eqref{disc2}. Finally, being $p>2$, then $r<p\frac{r+p-2}{p}$ and so we invoke \cite[Theorems $2.1$ \& $2.2$]{Po2} getting \eqref{decup}.

\end{proof}

\medskip

\begin{remark}[The critical case $ q=p-\frac{N}{N+1}$]
Some remarks on the critical growth case $q=p-\frac{N}{N+1}$ are in order to be given. Beyond assuming the Leray-Lions structure conditions in \eqref{A}  and the growth assumption \eqref{H0}, we  deal with this case taking into account initial data satisfying
\begin{equation*} 
u_0\in L^{1+\omega}(\Omega),\quad \om>0.
\end{equation*}
As already observed, such value of $q$ implies that the value $\si$ in \eqref{ID1} is equal to $1$. However, due to the criticality of this case, we have to ask for more than just $L^1(\Omega)$ data. In this way, we are allowed to consider solutions as in \cref{defrin1}, with \eqref{pot} replaced by
\[
(1+|u|)^{-\frac{1-\om}{p}}u\in L^p(0,T;W^{1,p}_0(\Omega)).
\]
Then, we proceed as before and we prove that \eqref{decup} holds with $\si=1+\om$.
\end{remark}

\section{The growth range with $L^1(\Omega)$ data}\label{L1reg}

\setcounter{equation}{0}
\renewcommand{\theequation}{\thesection.\arabic{equation}}

\numberwithin{equation}{section}

We now deal with the last case of superlinear growth \eqref{Q2}  that we recall being
\[
\frac{2N}{N+1}<p<N\quad\text{and}\quad\max\left\{\frac{p}{2},\frac{p(N+1)-N}{N+2} \right\}<q<p-\frac{N}{N+1}.
\]

We are going to deal with this case assuming $L^1(\Omega)$ data. In particular, since we can no longer require \eqref{pot}, we will ask for the asymptotic energy condition
\begin{equation}\label{ET}\tag{ET}
\lim_{n\to \infty}\frac{1}{n}\iint_{\{n\le |u|\le 2n\}}a(t,x,u,\N u)\cdot \N u=0.
\end{equation}

We consider the following notion of solution.
\begin{definition}\label{defrin2}
We say that a function $u\in {\mathcal{T}^{1,p}_0(Q_T)}$ is a solution of \eqref{P}  if satisfies \eqref{ET} and 
\begin{equation*}
H(t,x,\nabla u)\in L^1(Q_T), 
\end{equation*}
\begin{equation}\label{sr22}
\begin{array}{c}
\ds
-\integrale S(u_0)\vp (0,x)\,dx+\iint_{Q_T}-S(u)\vp_t+a(t,x,u,\N u)\cdot \N (S'(u)\vp)\,dx\,ds\\
[4mm]\ds
=\iint_{Q_T}H(t,x,\N u)S'(u)\vp\,dx\,ds
\end{array}
\end{equation}
for every $S\in W^{2,\infty}(\R)$ such that $S'(\cdot)$ has compact support and for every test function $\vp\in C_c^\infty([0,T)\times \Omega)$ such that $ S'(u)\vp\in L^p(0,T;W^{1,p}_0(\Omega))$ (i.e. $S'(u)\vp$ is equal to zero on $(0,T)\times\partial\Omega$).
\end{definition}

\medskip

The existence of solutions of \eqref{pb} has been proved in \cite[Theorem $6.5$]{M}.


\subsection{$L^1-L^1$ and Marcinkiewicz regularities }

\renewcommand{\theequation}{\thesection.\arabic{equation}}

\numberwithin{equation}{section}

As seen in \cref{Lsi}, the crucial step relies on a $\de$ argument which allows us to
move the attention from \eqref{P} to its "coercive version", i.e. \eqref{P} read in terms of $G_k(u)$. However, due to the low regularity of the initial data, we lose the purely contractive relation between $\|G_k(u(t))\|_{L^{1}(\Omega)}$ and $\|G_k(u_0)\|_{L^{1}(\Omega)}$. 

\medskip

\begin{lemma}\label{lemGkr2}
Assume  \eqref{ID2}, \eqref{A1}--\eqref{A2} with $p>\frac{2N}{N+1}$ and \eqref{H0} with \eqref{Q2}. Moreover let $u$ be a solution of \eqref{P}  in the sense of \cref{defrin2}. Then, for every $k>0$ so that
\[
\|G_k(u_0)\|_{L^{1}(\Omega)}< \de,
\]
where $\de>0$ is arbitrary fixed, we have
\begin{equation*}
\|G_k(u(t))\|_{L^{1}(\Omega)}<c\,\de^{\frac{1}{2}}\quad \forall  t\in[0,T],
\end{equation*}
for some positive constant $c$ depending on $|\Omega|$, $N$, $p$ and $q$.
\end{lemma}
\medskip
Before proving \cref{lemGkr2}, we recall some standard regularity results in renormalized settings with $L^1$-data.

\medskip

\begin{proposition}\label{marc}
Assume  \eqref{ID2}, \eqref{A1}--\eqref{A2} with $p>\frac{2N}{N+1}$ and \eqref{H0} with \eqref{Q2}. Moreover, let $u$ be a solution of \eqref{P}  in the sense of \cref{defrin2}. Then we have that  
\[
u\in C([0,T];L^1(\Omega))\cap M^{\frac{p(N+1)-N}{N}}(Q_T)
\]
and
\begin{equation}\label{mar}
|\N u|\in M^{\frac{p(N+1)-N}{N+1}}(Q_T).
\end{equation}
\end{proposition}
\begin{proof}
The Marcinkievicz regularities follow from \cite{BlMu,BlP}. 

As far as the continuity of $u(t)$ in $L^1(\Omega)$ is concerned, let $S'(u)\vp=\frac{T_\omega(G_k(u))}{\omega}$, $\vp=1$ and $\omega>0$, in \eqref{sr22}. Again, we note that such a test function can be made rigorous up to be multiplied by ${\theta_n}(G_k(u))$ and recalling the asymptotic condition \eqref{ET}.
Then the limit for $\omega\to 0$ provides us with the inequality
\begin{equation}\label{diss}
\integrale |G_k(u(t))|\,dx\le \integrale |G_k(u_0)|\,dx+\gamma\iint_{Q_T}|\N G_k(u)|^q\,dx\,dt.
\end{equation}
The gradient regularity in \eqref{mar} and \eqref{diss} imply that $\|G_k(u(t))\|_{L^1(\Omega)}\to 0$ when $k\to \infty$. Since we already know from \cite[Theorem $1.1$]{P1} that $T_k(u)\in C([0,T];L^1(\Omega))$ for every $k>0$, then
the Vitali-Lebesgue Theorem implies the continuity regularity $u\in C([0,T];L^1(\Omega))$.
\end{proof}

\begin{proof}[Proof of \cref{lemGkr2}]
We set 
\[
S_n'(u)\vp=\left[1-\frac{1}{(1+|G_k(u)|)^b}\right]\text{sign}(u)\q\t{with}\q b=\frac{(p-q)(N+1)}{N}-1,
\] 
$\vp=1$  in \eqref{sr22}. Note that $0<b<\frac{2}{N}$ by \eqref{Q2}. We justify the above choice reasoning as in \cref{marc}. Then, recalling \eqref{ET}, we get
\begin{equation}\label{in2}
\begin{array}{c}
\ds
\integrale S_n(u(t))\,dx+\al b\iint_{Q_t}\frac{|\N G_k(u)|^p}{(1+|G_k(u)|)^{b+1}}\,dx\,ds
\le \ga\iint_{Q_t}|\N G_k(u)|^qS_n'(u)\,dx\,ds\,.
\end{array}
\end{equation}
We are going to deal with the integral in the r.h.s.. An application of Young's inequality with indices $\left(\frac{p}{q},\frac{p}{p-q}\right)$ gives us
\begin{equation*}
\begin{array}{c}
\ds
\ga\iint_{Q_t}|\N G_k(u)|^qS_n'(u)\,dx\,ds\\
[4mm]
\ds
\le \frac{\al \,b}{2}\iint_{Q_t}\frac{|\N G_k(u)|^p}{(1+|G_k(u)|)^{b+1}}\,dx\,ds+c\iint_{Q_t}(1+|G_k(u)|)^{\frac{q(b+1)}{p-q}}\left(S_n'(u)\right)^{\frac{p}{p-q}}\,dx\,ds\\
[4mm]\ds
\le \frac{\al\,b}{2}\iint_{Q_t}\frac{|\N G_k(u)|^p}{(1+|G_k(u)|)^{b+1}}\,dx\,ds+c\iint_{Q_t}(1+|G_k(u)|)^{\frac{q(b+1)}{p-q}-1}|G_k(u)|\,dx\,ds,
\end{array}
\end{equation*}
being $\ds S_n'(u)<\frac{|G_k(u)|}{(1+|G_k(u)|)}$ because $b<1$ and for $c=c(\al,\ga,q,p,N)$. 
Then, we improve \eqref{in2} with
\begin{equation}\label{in3}
\begin{array}{c}
\ds
\integrale S_n(u(t))\,dx+\frac{\al\,b}{2}\iint_{Q_t}\frac{|\N G_k(u)|^p}{(1+|G_k(u)|)^{b+1}}\,dx\,ds
\\
[4mm]
\ds
\le c\iint_{Q_t}(1+|G_k(u)|)^{\frac{q(b+1)}{p-q}-1}|G_k(u)|\,dx\,ds
+\integrale |G_k(u_0)|\,dx.
\end{array}
\end{equation}
The choice of $b$ implies that $\frac{q(b+1)}{p-q}=q\frac{N+1}{N}$ which is, in particular, the Gagliardo-Nirenberg exponent of the spaces
\[
L^\infty(0,T;L^1(\Omega))\cap  L^q(0,T;W_0^{1,q}(\Omega)). 
\]
Since \cref{marc} provides us with the above regularities, we are allowed to consider the limit on $n\to\infty$ in \eqref{in3} getting
\begin{equation}\label{ooo}
\begin{array}{c}
\ds
\integrale S(u(t))\,dx+\frac{\al\,b}{2}\iint_{Q_t}\frac{|\N G_k(u)|^p}{(1+|G_k(u)|)^{b+1}}\,dx\,ds
\\
[4mm]
\ds
\le c\iint_{Q_t}(1+|G_k(u)|)^{\frac{q(b+1)}{p-q}-1}|G_k(u)|\,dx\,ds
+\integrale |G_k(u_0)|\,dx,
\end{array}
\end{equation}
where $S(v)=\int_0^v 1-\frac{1}{(1+|y|)^b}\,dy$.\\
The above estimate implies that the l.h.s. of \eqref{ooo} is bounded and, in particular, that  
$$
(1+|G_k(u)|)^{\frac{-(b+1)}{p}}|G_k(u)| \in L^\infty(0,T;L^{\frac{p}{p-1-b}}(\Omega))\cap L^p(0,T;W_0^{1,p}(\Omega)).
$$
Since 
\[
\frac{p}{p-1-b}<p\frac{N+\frac{p}{p-1-b}}{N}<p^*\qq\t{and}\qq p<p\frac{N+\frac{p}{p-1-b}}{N}
\]
thanks to \eqref{Q2} and the definition of $b$, 
we invoke again Gagliardo-Nirenberg regularity results, obtaining the regularity
$$(1+|G_k(u)|)^{\frac{-(b+1)}{p}}|G_k(u)|\in L^{\lm}(Q_T)\q \t{where}\q \lm=p\frac{N+\frac{p}{p-1-b}}{N}.$$
In particular, the related Gagliardo-Nirenberg inequality can be estimated as
\begin{equation}\label{gn}
\ds \int_0^T\|(1+|G_k(u)|)^{\frac{-(b+1)}{p}}|G_k(u)|\|_{L^{\lm}(\Omega)}^{\lm}\,dt\\
\le c_{GN} \|G_k(u)\|_{L^\infty(0,T;L^{1}(\Omega))}^{\frac{p}{N}}\iint_{Q_T}\frac{|\N G_k(u)|^p}{(1+|G_k(u)|)^{b+1}}\,dx\,dt. 
\end{equation}
Let us come back to \eqref{in3}. 
Since $$ \frac{q(b+1)}{p-q}=\left(-\frac{(b+1)}{p}+1\right)\lm$$
by definitions of $b$ and $\lm$, we estimate the r.h.s. of \eqref{in3} taking advantage of \eqref{gn} as follows:
\begin{equation*}
\ga\iint_{Q_t}|\N G_k(u)|^qS_n'(u)\,dx\,ds\le \ga\, c_{GN}\|G_k(u)\|_{L^\infty(0,T;L^{1}(\Omega))}^{\frac{p-q}{N}}\iint_{Q_T}\frac{|\N G_k(u)|^p}{(1+|G_k(u)|)^{b+1}}\,dx\,dt.
\end{equation*}
We thus deduce
\begin{equation*}
\begin{array}{c}
\ds
\integrale S(u(t))\,dx
+\al\,b\iint_{Q_t}\frac{|\N G_k(u)|^p}{(1+|G_k(u)|)^{b+1}}\,dx\,ds\\
[4mm]
\ds
\le \ga\, c_{GN}\|G_k(u)\|_{L^\infty(0,t;L^{1}(\Omega))}^{\frac{p-q}{N}}\iint_{Q_t}\frac{|\N G_k(u)|^p}{(1+|G_k(u)|)^{b+1}}\,dx\,ds
+\integrale |G_k(u_0)|\,dx,
\end{array}
\end{equation*}
where the limit on $n\to\infty$ has be taken too.\\

\noindent
\textit{The $\de$ argument.}\\
We observe that $S(v)\ge c_1\min\{v,v^2\}$, where the constant $c_1>1$ depends only on $N$, $p$ and $q$.\\
Then, we proceed as in \cref{lemGkp} (see the $\de$ argument) fixing a small value $\de_0$ so that the inequality
$\ds \al\,b-\ga\,c_{GN}(c_0\de^\frac{1}{2})^{\frac{p-q}{N}}>0$ holds for $ c_0=2\max\left\{1/c_1,\left(|\Omega|/c_1\right)^\frac{1}{2} \right\}$ and $\de<\de_0$.
Moreover, we let $k_0$ large enough so that
\begin{equation}\label{**}
\|G_k(u_0)\|_{L^1(\Omega)}<\de\qquad\forall  k\ge k_0
\end{equation}
and define
\[
T^*:=\sup\{\tau> 0: \,\|G_k(u(s))\|_{L^{1}(\Omega)}\le c_0\de^\frac{1}{2}, \,\,\forall  s\le \tau  \} \qquad\forall k\ge k_0.
\]
Notice that $T^*>0$ thanks to the continuity result proved in \autoref{marc}. The above choice of $\de$ and \eqref{**} imply
\begin{equation*}
\integrale S(G_k(u(t)))\,dx\le \integrale |G_k(u_0)|\,dx<\de\qquad\forall k\ge k_0,\,\,\forall t\le T^*.
\end{equation*}
Therefore, by definition of $c_1$ and $c_0$, we obtain
\begin{equation}\label{contr}
\begin{split}
\integrale |G_k(u(t))|\,dx & \le \int_{\{ |G_k(u(t))|>1\}}|G_k(u(t))|\,dx+|\Omega|^{\frac{1}{2}}\left( \int_{\{|G_k(u(t))|\le 1\}} |G_k(u(t))|^2\,dx\right)^\frac{1}{2}\\
&\le \frac{1}{c_1}\integrale S(G_k(u(t)))\,dx+\left(\frac{|\Omega|}{c_1}\right)^{\frac{1}{2}}\left(\integrale S(G_k(u(t)))\,dx\right)^{\frac{1}{2}}\\
&< c_0\de^\frac{1}{2}
\end{split}
\end{equation}
for every $t\le T^*$. Finally, a contradiction argument extends the inequality in \eqref{contr} to the whole time interval.
\end{proof}
\medskip
\begin{remark}\label{contr2}
Again, if $u_0\in L^\infty(\Omega)$, then the $\de$ argument provides us with a contraction result in $L^\infty(\Omega)$ (see also \cref{corinftyp}).
\end{remark}

\subsection{The regularizing effect $L^1-L^r$ and long time decays}

\renewcommand{\theequation}{\thesection.\arabic{equation}}

\numberwithin{equation}{section}

\begin{proposition}\label{Proprin1}
Assume  \eqref{ID2},  \eqref{A1}--\eqref{A2} with $p>\frac{2N}{N+1}$ and \eqref{H0} with \eqref{Q2}. Moreover, let $u$ be a solution of \eqref{P}  in the sense of \cref{defrin2}. Then the claim of \cref{Prop2rp} holds true.
\end{proposition}
\begin{proof}
In order to reason as in \cref{Prop2rp}, we assume \eqref{S2} and modify \eqref{S1} asking for a function $S(\cdot)$ satisfying
\begin{equation}\label{S11}
S''(v)\le L(1+|v|)^{-(2+b)}|v|,\qquad b=(p-q)\frac{N+1}{N}-1.
\end{equation}
In this way, we can repeat the argument at the very beginning of the proof of\cref{Prop2rp}, getting \eqref{w11rin}. In particular, \cref{lemGkr2} provides us with an inequality as in \eqref{phi}.\\
We conclude exhibiting a function which, for fixed $n$, verifies both \eqref{S2} and \eqref{S11}:
\[
S'(v)=\int_0^v (1+|y|)^{-(b+2)}|y| T_n(y)^{r-1+b}\,dz,\quad r>1.
\] 
\end{proof}

We conclude this Section observing that, once we have the contraction result of \cref{contr2} as well as the regularizing effect provided by \cref{Proprin1}, then we are allowed to reason as in \cref{asfinen} getting the same long time decays results as in \cref{decinf}.

\section{Further comments}

\subsection{On the notion of solution}

We here point out that we could consider different notions of solutions than \cref{defrin1,defrin2}. Indeed, as shown in \cite{MP}, the \cref{defloc,defrc,defloc2} below are strictly related to the ones previously considered.

\medskip

\begin{definition}\label{defloc} A  function $u$ is a solution to \eqref{P}  if $u(0)= u_0$,
$$u\in C([0,T];L^\si(\Omega))\cap L^p_{loc}(0,T;W^{1,p}_0(\Omega))$$  and $u$ satisfies the weak formulation
\begin{equation*}
  \iint_{Q_T} - u\varphi_t+ a(t,x,u,\N u)\cdot  \N\varphi\,dx\,dt =  \iint_{Q_T} H(t,x, \N u)\varphi\,dx\,dt
\end{equation*}
for every $\varphi\in C^\infty_c((0,T)\times \Omega)$.
\end{definition}


\medskip

\begin{definition}\label{defrc}
A function $u$ is a solution of \eqref{P}  if the regularity condition \eqref{pot} holds, $H(t,x,\N u)\in L^1(Q_T)$ and $u$ satisfies the weak formulation
\[
-\int_{\Omega}u_0\varphi(0)\,dx+\iint_{Q_T}[-\varphi_t u +a(t,x,u,\N u)\cdot  \N \varphi]\,dx\,dt =\iint_{Q_T}H(t,x, \N u)\varphi \,dx\,dt
\]
for every   $\vp\in C^\infty_c([0,T)\times \Omega)$.
\end{definition}

\medskip

\begin{proposition}
We have that \cref{defloc} is equivalent to
\begin{itemize}
\item \cref{defrc} if the gradient growth occurs with rates as in  \eqref{fe} (i.e., we consider the subinterval of \eqref{Q1}  such that $\si\ge 2$);
\item \cref{defrin1} if the gradient growth occurs with rates as in \eqref{Q1}.
 \end{itemize}
\end{proposition}

\medskip

\begin{definition}\label{defloc2} A  function $u$ is a solution to \eqref{P}  if $u(0)= u_0$,
$$u\in C([0,T];L^1(\Omega))\cap L^p_{loc}(0,T;W^{1,p}_0(\Omega))$$  and $u$ satisfies the weak formulation
\begin{equation*}
  \iint_{Q_T} - u\varphi_t+ a(t,x,u,\N u)\cdot  \N\varphi\,dx\,dt =  \iint_{Q_T} H(t,x, \N u)\varphi\,dx\,dt
\end{equation*}
for every $\varphi\in C^\infty_c((0,T)\times \Omega)$.
\end{definition}


\medskip

\begin{proposition}
We have that \cref{defloc2} is equivalent to \cref{defrin2}.
\end{proposition}

\medskip

We omit the proof of the Propositions above since they are a simple generalisation of  \cite[Propositions $2.2$, $2.4$ and $4.2$]{MP}.


\begin{thebibliography}{99}


\bibitem{ADAP}B. Abdellaoui, A. Dall'Aglio \& I. Peral, \href{http://www.sciencedirect.com/science/article/pii/S0021782408000470?via%3Dihub}{\it Regularity and nonuniqueness results for parabolic problems arising in some physical models, having natural growth in the gradient}, Adv. Nonlin. Stud., Vol. 90 Issue 3 (2008), pp. 242-269.




\bibitem{BarPo} G. Barles \& A. Porretta, \href{http://www.mat.uniroma2.it/~porretta/papers/Barles-Porretta-AnnPisa.pdf}{\it Uniqueness for unbounded solutions to stationary viscous
Hamilton-Jacobi equations}, Ann. Scuola Norm. Sup. di Pisa Cl. Sci., Vol. 5 (2006), pp. 107-136.

\bibitem{BASW} M. Ben-Artzi, P. Souplet \& F. Weissler, \href{http://www.sciencedirect.com/science/article/pii/S0021782401012430}{\it The local theory for viscous Hamilton-Jacobi equations in Lebesgue spaces}, J. Math. Pures Appli., Vol. 81 (2002), pp. 343-378.

\bibitem{BD}  S. Benachour \& S. Dabuleanu, \href{https://projecteuclid.org/euclid.ade/1355867980}{\it The mixed Cauchy-Dirichlet problem for a viscous Hamilton-Jacobi equation}, Adv. Diff. Eq., Vol. 8  (2003), pp. 1409-1452.

\bibitem{BDL} S. Benachour, S. Dabuleanu \& P. Lauren\c{c}ot, \href{https://hal.archives-ouvertes.fr/hal-00093133/file/s5bsdhphl.pdf}{\it Decay estimates for a viscous Hamilton-Jacobi equation with homogeneous Dirichlet boundary conditions}, Asymptot. An., Vol. 51 (2007), pp. 209-229.


\bibitem{BDGM}L. Boccardo, J.I. Diaz, D. Giachetti \& F. Murat, {\it Existence and regularity of
renormalized solutions for some elliptic problems involving derivatives of
nonlinear terms}, J. Diff. Eq. Vol. 106 (1993),  215-237.



\bibitem{BlMu} D. Blanchard \& F. Murat, \href{https://www.cambridge.org/core/journals/proceedings-of-the-royal-society-of-edinburgh-section-a-mathematics/article/renormalised-solutions-of-nonlinear-parabolic-problems-with-l1-data-existence-and-uniqueness/9402C539D76E5F35F55ACBBBB48F1F48}{\it Renormalised solutions of nonlinear parabolic problems with $L^1$ data: existence and uniqueness}, Proceedings of the Royal Society of Edinburgh, Vol. 127 (1997), pp. 1137-1152.

\bibitem{BlP} D. Blanchard \& A. Porretta, \href{http://archive.numdam.org/ARCHIVE/ASNSP/ASNSP_2001_4_30_3-4/ASNSP_2001_4_30_3-4_583_0/ASNSP_2001_4_30_3-4_583_0.pdf}{\it Nonlinear parabolic equations with natural growth terms and measure initial data}, Ann. Scuola Norm. Sup. Pisa Classe di Scienze $4^e$ s\'{e}rie, Vol. 30 (2001), pp. 583-622.


 

\bibitem{BrCa}H. Brezis \& T. Cazenave, \href{http://link.springer.com/article/10.1007%2FBF02790212}{\it A nonlinear heat equation with singular initial data}, J. Anal. Math., Vol. 68 (1996), pp. 277-304.

\bibitem{BrCaYvRa} H. Brezis, T. Cazenave, Y. Martel \& A. Ramiandrisoa, \href{http://projecteuclid.org/euclid.ade/1366896315}{\it Blow up for $u_t-\D u=g(u)$ revisited}, Adv. Diff. Eq., Vol. 1, Issue 1, (1996), pp. 73-90. 

\bibitem{CG} F. Cipriani \& G. Grillo, \href{http://www.sciencedirect.com/science/article/pii/S0022039600939858}{\it Uniform bounds for solutions to quasilinear parabolic equations}, J. Diff. Eq., Vol. 177 (2001), pp. 209-234.

\bibitem{DMOP} G. Dal Maso, F. Murat, L. Orsina \& A. Prignet, \href{http://archive.numdam.org/ARCHIVE/ASNSP/ASNSP_1999_4_28_4/ASNSP_1999_4_28_4_741_0/ASNSP_1999_4_28_4_741_0.pdf}{\it Renormalized solutions of elliptic equations with general measure data}, Ann. Scuola Norm. Sup. Pisa, Vol. 28 (1999), pp. 741-808.

\bibitem{Di} E. DiBenedetto, {Degenerate parabolic equations}, Springer-Verlag.

\bibitem{Fri}  A. Friedman, {Partial Differential Equations}, R. E. Krieger Pub. Co., (1976).


\bibitem{GMP} N. Grenon, F. Murat \& A. Porretta, \href{http://annaliscienze.sns.it/index.php?page=Article&id=300}{\it A priori estimates and existence for elliptic equations with gradient dependent terms}, Ann. Sc. Norm. Super. Pisa Cl. Sci., Vol. 13 (2014), pp. 137-205.


\bibitem{LeM} T. Leonori \& M. Magliocca, \href{https://arxiv.org/pdf/1711.06634.pdf}{Comparison results for unbounded solutions for a parabolic Cauchy-Dirichlet problem with superlinear gradient growth}, available on arXiv.


\bibitem{M} M. Magliocca, \href{http://www.sciencedirect.com/science/article/pii/S0362546X17302390}{{\it Existence results for a Cauchy-Dirichlet parabolic problem with a repulsive gradient term}}, Nonlin. Anal., Vol. 166  (2018), pp. 102-143.

\bibitem{MP} M. Magliocca \& A. Porretta, \href{https://arxiv.org/pdf/1707.01761.pdf}{Local and global time decay for parabolic equations with super linear first order terms}, available on arXiv.

\bibitem{Mu} F. Murat, \href{http://www.cimpa-icpam.org/anciensite/NotesCours/PDF/2009/Alexandrie_Murat_2.pdf}{\it Soluciones renormalizadas de EDP ellipticas non lineales}, Laboratoire d'Analyse Num\'{e}rique, Paris VI, Technical Report R93023, (1993).

\bibitem{P1} A. Porretta, \href{http://link.springer.com/article/10.1007/BF02505907}{\it Existence Results for Nonlinear Parabolic Equations via Strong Convergence of Truncations}, Ann. Mat. pura e appl., Vol. 177  (1999), pp.143-172.

\bibitem{P2} A. Porretta, \href{https://projecteuclid.org/euclid.ade/1357141502}{\it Local existence and uniqueness of weak solutions for nonlinear parabolic equations with superlinear growth and unbounded initial data}, Adv. Diff. Eq., Vol. 6 (2001), pp. 73-128.

\bibitem{Po3} A. Porretta, \href{}{\it On the comparison principle for p-Laplace type operators with first
order terms}, in "On the notions of solution to nonlinear elliptic problems: results
and developments", Quad. Mat. 23, Dept. Math., Seconda Univ. Napoli, Caserta, (2008), pp. 459-497.

\bibitem{PZ} A. Porretta \& E. Zuazua, \href{http://www.sciencedirect.com/science/article/pii/S0294144911001028}{\it Null controllability of viscous Hamilton-Jacobi equations}, Ann. I.H.P. Analyse Nonlin\`{e}aire , Vol. 29 (2012), pp. 301-333.

\bibitem{Po2} M. M. Porzio, \href{http://link.springer.com/article/10.1007/s00028-009-0024-8}{\it On decay estimates}, J. Evol. Eq. Vol. 9 (2009), pp. 561-591.

\bibitem{Po3} M. M. Porzio, \href{http://www.sciencedirect.com/science/article/pii/S0362546X11003051}{\it Existence, uniqueness and behavior of solutions for a class of nonlinear
parabolic problems}, Nonlin. An., Vol. 74 (2011), pp. 5359-5382.


\end{thebibliography}
\end{document}